\documentclass[a4paper, 11pt]{amsart}
\usepackage[margin=1in]{geometry}
\usepackage[T1]{fontenc}
\usepackage[utf8]{inputenc}
\usepackage{lmodern}
\usepackage{amsmath,amssymb,amsthm,mathtools,mathrsfs}
\usepackage{enumitem}
\usepackage{booktabs,longtable,array}
\usepackage{microtype}
\usepackage{tikz}
\usepackage[colorlinks=true,linkcolor=blue,citecolor=blue,urlcolor=blue]{hyperref}

\makeatletter
\@namedef{subjclassname@2020}{\textup{2020} Mathematics Subject Classification}
\makeatother

\numberwithin{equation}{section}

\newtheorem{theorem}{Theorem}[section]
\newtheorem{proposition}[theorem]{Proposition}
\newtheorem{lemma}[theorem]{Lemma}
\newtheorem{corollary}[theorem]{Corollary}
\theoremstyle{definition}
\newtheorem{definition}[theorem]{Definition}

\newtheorem{remark}[theorem]{Remark}
\newtheorem{convention}[theorem]{Convention}

\newtheorem*{mainthmA}{Theorem A}
\newtheorem*{mainthmB}{Theorem B}

\DeclareMathOperator{\red}{red}

\DeclareMathOperator{\ord}{ord}

\DeclareMathOperator{\val}{val}
\DeclareMathOperator{\supp}{supp}

\newcommand{\an}{\mathrm{an}}

\newcommand{\PP}{\mathbf P}

\newcommand{\R}{\mathbb R}

\newcommand{\cO}{\mathcal O}

\newcommand{\cX}{\mathcal X}
\newcommand{\cY}{\mathcal Y}

\title[Nevanlinna theory on Berkovich curves]{Nevanlinna Theory on Berkovich Curves}
\author[T. B. Cao \and J. H. Peng]{ Tingbin Cao\and Jiahu Peng}
	
	\address[Tingbin Cao and Jiahu Peng]{Department of Mathematics, Nanchang University, Nanchang city, Jiangxi 330031, P. R. China}
	\email{tbcao@ncu.edu.cn, pjh@email.ncu.edu.cn}
	
	\thanks{This work is supported by the National Natural Science Foundation of China (\#12571082), the Jiangxi Natural Science Foundation (\#20232ACB201005) and the Shandong Natural Science Foundation (\#ZR2024MA024)}
	\date{}

\begin{document}


\begin{abstract}We develop a value distribution theory for finite analytic morphisms
between proper connected quasi-smooth Berkovich curves over a complete
algebraically closed non-Archimedean field.  The central tool is a
multi-tail Dirichlet Green function on finite truncations of a
generalized skeleton, which makes non-Archimedean potential theory on
skeleta compatible with global Nevanlinna asymptotics.  Using it, we
prove the first main theorem as an exact identity relating the
characteristic, counting, and proximity functions.  For logarithmic
targets we establish a second main theorem in both ramified and
level-one truncated forms. The defect relation is also given.  The theory provides a
Berkovich-analytic framework for non-Archimedean value distribution,
bridging classical Nevanlinna theory and its $p$-adic counterpart.
\end{abstract}

\maketitle
\tableofcontents

\section{Introduction and Main Results}\label{sec:introduction}

\subsection{Background and motivation}

Classical Nevanlinna theory, initiated by R.~Nevanlinna
\cite{Nevanlinna1936} and developed systematically in
\cite{Hayman1964,CherryYe2001,LangCherry1990,Ru2001}, studies the value
distribution of meromorphic functions $f$ through the counting function
$N(r,f)$, the proximity function $m(r, f)$, and the characteristic function
$T(r, f)$.  The extension of this theory to holomorphic mappings between
Riemann surfaces was accomplished by Chern \cite{Chern1960a}, who obtained
Nevanlinna-type basic theorems in that setting.  In a companion paper,
Chern \cite{Chern1960b} proved the integrated form of the first main
theorem for complex analytic mappings in several complex variables,
recasting the first main theorem as a potential-theoretic identity.
Further work of Bott and Chern \cite{BottChern1965} extended these ideas
to Hermitian vector bundles and the equidistribution of the zeroes of
their holomorphic sections.

In the non-Archimedean setting, a $p$-adic Nevanlinna theory was
initiated by  Hà Huy Khoái
\cite{Khoai1983,KhoaiSurvey} and Boutabaa \cite{Boutabaa1990,Boutabaa1991}, and further developed by Boutabaa-Escassut\cite{BoutabaaEscassut2001}, Escassut
\cite{Escassut2016}, Hu and Yang \cite{HuYang1999}, Cherry-Ye
\cite{CherryYe1997}, Ru \cite{Ru-2001} and others.  This theory successfully transplanted
the formal structure of the classical theory to valued fields: counting,
proximity, and characteristic functions were defined, and first and second
main theorems were established.  However, the classical $p$-adic approach
operates directly on the valued field and inherits several structural
limitations from the non-Archimedean absolute value.  First, in the
meromorphic case the Hadamard estimate contains an infinite term on the
right-hand side, so it does not yield a non-trivial bound on the number of
zeros of $f$ \cite[pp.~2--3]{KhoaiSurvey}.  Second, the characteristic
function $T(r)$ is constructed from the absolute value and lacks a direct
interpretation in terms of curvature or potential-theoretic kernels,
in contrast with the complex case where $T(r)$ arises from the Jensen
formula and the Green kernel of the disk.  Third, and most
fundamentally, the $p$-adic theory provides no skeletal or
graph-theoretic reduction: one works on the valued field itself, and
therefore the Laplacian, Green kernels, and equidistribution machinery
available in the complex setting have no direct counterpart. In short, the classical $p$-adic approach equips the valued field with a value-distribution formalism, but not with the geometric and potential-theoretic structures that make the complex theory effective.

More recent developments in Berkovich analytic geometry, especially the
theory of skeleta and non-Archimedean potential theory
\cite{Thuillier,BakerRumelyBook,BakerIntro,BPR,CLD,GublerForms}, make it
possible to formulate value distribution in a global potential-theoretic
manner that addresses these limitations.  In this framework, a smooth
Berkovich curve carries a natural potential theory analogous to that of
a Riemann surface \cite{Thuillier}: there is a Laplacian compatible with
graph Laplacians on skeleta \cite{Thuillier,CTT}, a theory of real
differential forms and currents \cite{CLD,GublerForms}, and well-behaved
harmonic functions \cite{BakerRumelyBook,BakerIntro}.  The curve admits
a semistable skeleton that reduces many global questions to finite metric
graphs, where Green kernels, Laplacians, and integration by parts are
available. The present paper develops a non-Archimedean Nevanlinna theory within this
skeletal potential-theoretic framework.

\subsection{Setting and notation}

Let $K$ be an algebraically closed, complete, nontrivially valued
non-Archimedean field.  We work on a proper connected quasi-smooth
strictly $K$-analytic Berkovich curve $\cX$.  Fix a nonempty finite
puncture set
\[
        H\subset\cX(K),
\]
identified with its reduced effective Cartier divisor, and put
\[
        U:=\cX\setminus H,
        \qquad
        n_H:=\#H=\deg H.
\]
We fix a generalized skeleton
\[
        \Sigma=\Gamma_0\cup\bigcup_{h\in H}\rho_h\subset U
\]
with retraction
\[
        \tau:U\longrightarrow\Sigma,
\]
where $\Gamma_0$ is a compact core and each $\rho_h$ is the unbounded ray
corresponding to $h\in H$.  For $R>0$, let
\[
        \Sigma_R=\Gamma_0\cup\bigcup_{h\in H}[v_h,\xi_h(R)]
\]
be the finite truncation, with boundary
\[
        \partial\Sigma_R=\{\xi_h(R):h\in H\}.
\]
Let
\[
        \tau_R:U\longrightarrow\Sigma_R
\]
be the finite-truncation retraction.

The central tool is the multi-tail Dirichlet Green function $G_R$ on
$\Sigma_R$.  It is the unique continuous piecewise-affine function
satisfying
\[
        G_R=0\quad\text{on }\partial\Sigma_R,
        \qquad
        G_R\text{ is harmonic on }
        \Sigma_R\setminus(\{b\}\cup\partial\Sigma_R),
\]
and whose graph Laplacian at the fixed base point $b\in\Gamma_0$ has
coefficient $-n_H$.  Its boundary measure is
\[
        \nu_R:=\Delta_{\Sigma_R}G_R+n_H\delta_b
        =\sum_{h\in H}\lambda_h(R)\delta_{\xi_h(R)}.
\]
One has $\lambda_h(R)\ge0$, $\sum_{h\in H}\lambda_h(R)=n_H$, and
\[
        \lambda_h(R)=1+O(R^{-1})
        \qquad(R\to\infty).
\]
Moreover, on every fixed compact subgraph $C\subset\Sigma$,
\[
        G_R(x)=R+O_C(1)
\]
uniformly for $x\in C$.  These properties are proved in
Section~\ref{sec:multi-tail-kernel}.

For a finite analytic morphism
\[
        f:\cX\longrightarrow\cY
\]
between proper connected quasi-smooth strictly $K$-analytic Berkovich
curves, and for an effective analytic Cartier divisor $D$ on $\cY$, we
define the characteristic, counting, and proximity functions
\[
        T_f(D,R),\qquad N_f(D,R),\qquad m_f(D,R)
\]
using the finite-truncation Green kernel $G_R$.  Their precise
definitions are given in
Definition~\ref{def:morphism-nevanlinna-functions}.  In the
projective-line case $\cY=\PP_K^{1,\an}$ we use the standard shorthand
\[
        T_f(R):=T_f\bigl(\cO_{\PP_K^{1,\an}}(1),R\bigr),
\]
and, for $a\in\mathbf P^1(K)$,
\[
        N_f(a,R):=N_f([a],R),
        \qquad
        m_f(a,R):=m_f([a],R).
\]

\subsection{Main results}
Our first main result is an exact identity relating the characteristic,
counting, and proximity functions, called the first main theorem in Nevanlinna theory. This
identity is a non-Archimedean and Berkovich-analytic analogue in
dimension one of Chern's integrated first main theorem
\cite{Chern1960b}.

Let $f:\cX\to\cY$ be a nonconstant finite analytic morphism between
proper connected quasi-smooth strictly $K$-analytic Berkovich curves.
Let $D$ be an effective analytic Cartier divisor on $\cY$ supported on
$\cY(K)$, and let $s_D$ be the canonical section of
$\cO_{\cY}(D)$.  Fix a piecewise-linear metric on $\cO_{\cY}(D)$ whose
pullback to
\[
        L_{f,D}:=f^*\cO_{\cY}(D)
\]
is tail-normalized admissible relative to $(\Sigma,H)$, and put
\[
        u_{f,D}:=(-\log\|f^*s_D\|)|_\Sigma.
\]

\begin{mainthmA}(Theorem~\ref{thm:morphism-fmt})
There exists $R_0(D)>0$ such that, for $R\ge R_0(D)$,
\[
        T_f(D,R)=N_f(D,R)+m_f(D,R)-n_Hu_{f,D}(b).
\]
\end{mainthmA}

The next main result is an estimation of the characteristic by counting functions, called the second main theorem in Nevanlinna theory. 

Below assume that $f:\cX\to\cY$ is a nonconstant finite generically \'etale analytic morphism between
proper connected quasi-smooth strictly $K$-analytic Berkovich curves. Write
$K_{\cX}:=\Omega^1_{\cX/K}$ and $K_{\cY}:=\Omega^1_{\cY/K}$
for the canonical line bundles. Let
\[
        J_f\in H^0\bigl(\cX,K_{\cX}\otimes f^*K_{\cY}^{-1}\bigr)
\]
be the analytic Jacobian section of $f$, and let
\[
        R_f:=\operatorname{div}(J_f)
\]
be its Cartier ramification divisor.  For a reduced effective divisor
$D$ on $\cY$, define the logarithmic Jacobian section
\[
        \sigma_{f,D}:=\frac{f^*s_D}{J_f},
\]
viewed as a meromorphic section of
\[
        L^{\log}_{f,D}:=f^*(K_{\cY}+D)\otimes K_{\cX}^{-1}.
\]
It follows from Lemma \ref{lem:div-sigma} that its divisor is
\[
        \operatorname{div}(\sigma_{f,D})=f^*D-R_f.
\]

\begin{mainthmB}(Theorems~\ref{thm:target-curve-SMT})
For sufficiently large $R$, we have
\[
        T_f(K_{\cY}+D,R)-T_{K_{\cX}+H}(R)+N_{R_f}(R)
        \leq N_f(D,R)+O(1).
\]
Consequently, in level-one truncated form,
\[
        T_f(K_{\cY}+D,R)-T_{K_{\cX}+H}(R)
        \leq N_f^{[1]}(D,R)+O(1),
\]
where
\[
        N_f^{[1]}(D,R)
        :=\int_{\Sigma_R}G_R\,
        d\tau_{R*}\bigl((f^*D|_U)_{\mathrm{red}}\bigr)
\]
is the level-one counting function.
\end{mainthmB}

\subsection{Methods and organization}

The methods of the paper combine several strands of non-Archimedean
potential theory.  From Thuillier's theory we use the compatibility
between the analytic Laplacian and graph Laplacians on skeleta, together
with the Poincar\'e--Lelong formula \cite{Thuillier}.  The
coordinate-free Poincar\'e--Lelong formalism of Chambert-Loir and
Ducros \cite{CLD}, and the formal metric theory of Vilsmeier
\cite{Vilsmeier}, provide the framework for metrized analytic line
bundles and their skeletal curvature measures.  The ramification theory
of Cohen-Temkin-Trushin \cite{CTT} supplies the local differential
index and the degree form of Riemann-Hurwitz.  Finally, the
graph-theoretic Green kernels and Laplacians are those of Baker-Rumely \cite{BakerRumelyGraphs}
and Baker-Faber \cite{BakerFaberGraphs}.  The new
ingredient is the finite-truncation multi-tail Dirichlet kernel, which
makes the potential theory on generalized skeleta compatible with
global Nevanlinna asymptotics.

The paper is organized as follows.
Section~\ref{sec:prelim} recalls the necessary Berkovich-analytic
preliminaries, skeleta, formal metrics, and metric-graph potential
theory.  Section~\ref{sec:multi-tail-kernel} constructs finite
truncations and the multi-tail Dirichlet Green kernels, and proves their
existence, positivity, and asymptotics.
Section~\ref{sec:first-main-dirichlet} establishes the first main
theorem for morphisms of Berkovich curves.  Section~\ref{sec:second-main}  introduces target divisors,
line-bundle characteristics, the Jacobian section, and the logarithmic
Jacobian. The finite-truncation
ramification identity and the second main theorem are obtained. Our Second main theorem derives a defect relation.

\section{Preliminaries for value distribution on Berkovich curves}\label{sec:prelim}

\subsection{Non-Archimedean absolute values, valuation rings, and residue fields}\label{subsec:valued-fields}

\begin{definition}[Non-Archimedean absolute value]\label{def:nonarch-abs}
Let $K$ be a field. An absolute value on $K$ is a map
\[
        |\cdot|:K\longrightarrow\R_{\geq 0}
\]
such that, for all $a,b\in K$,
\[
 |a|=0\Longleftrightarrow a=0,\qquad
 |ab|=|a||b|,\qquad
 |a+b|\leq |a|+|b|.
\]
It is non-Archimedean if it also satisfies the ultrametric inequality
\begin{equation}\label{eq:ultrametric-inequality}
        |a+b|\leq\max\{|a|,|b|\}.
\end{equation}
It is nontrivial if $|a|\neq 1$ for some $a\in K^\times:=K\setminus\{0\}$. (See \cite[Chapters~1-2]{RobertPadic} and
\cite[Chapter~1]{BGR}).
\end{definition}

\begin{lemma}\label{lem:unequal-norm}\cite[Chapter~1]{BGR}.
If $a,b\in K$ satisfy $|a|\neq |b|$, then
\[
        |a+b|=\max\{|a|,|b|\}.
\]
\end{lemma}

We use the associated additive valuation
\begin{equation}\label{eq:additive-valuation}
 \val(a):=-\log|a|\quad(a\in K^\times),
 \qquad \val(0):=+\infty.
\end{equation}
It satisfies
\[
 \val(ab)=\val(a)+\val(b),\qquad
 \val(a+b)\geq\min\{\val(a),\val(b)\}.
\]
Multiplicative radii and additive length parameters are related by
\[
        t=-\log r.
\]
We use this additive parameter both for lengths on annular skeleta and for
the variables of piecewise-affine functions.
The multiplicative value group of $(K,|\cdot|)$ is
\[
        |K^\times|:=\{|a|:a\in K^\times\}\subset\mathbb R_{>0}.
\]

\begin{definition}[Valuation ring, maximal ideal, and residue field]\label{def:valuation-residue}
Define
\[
 K^\circ:=\{a\in K:|a|\leq 1\},\qquad
 K^{\circ\circ}:=\{a\in K:|a|<1\}.
\]
Then $K^\circ$ is a local ring with unique maximal ideal
$K^{\circ\circ}$. The quotient field
\[
        \widetilde K:=K^\circ/K^{\circ\circ}
\]
is the residue field of $K$. The image of $a\in K^\circ$ is denoted by
$\widetilde a\in\widetilde K$. (See \cite[Chapter~1]{BGR} and
\cite[Chapter~2]{RobertPadic}).
\end{definition}

Unless otherwise stated, $K$ is complete, algebraically closed, and
nontrivially valued. Completeness keeps limits of convergent analytic series
inside $K$. Algebraic closedness permits closed points, zeros, and poles to be
described as $K$-points; it also implies that the residue field
\(\widetilde K\) is algebraically closed \cite[Chapter~3]{BGR}. Further background appears in
\cite[Sections~2.1-2.4]{RobertPadic} and \cite[Section~2]{TemkinIntro}.

\subsection{Analytic curves}\label{subsec:analytic-standing}

Throughout, $\cX$ denotes a proper connected quasi-smooth
strictly $K$-analytic Berkovich curve. A strict $K$-affinoid algebra is a quotient of a convergent power-series
algebra $K\{T_1,\ldots,T_n\}$, endowed with its quotient Banach topology; a
strictly $K$-analytic curve is locally covered by spectra of one-dimensional
strict $K$-affinoid algebras.  Quasi-smoothness is the analytic smoothness
condition for curves.  In particular, $\Omega^1_{\cX/K}$ is locally free of
rank one, and the local ring at a $K$-rational point is a one-dimensional
regular local ring; see \cite{Berkovich,TemkinIntro}. Properness implies
compactness and empty analytic boundary.

For meromorphic functions, we use Thuillier's compatibility between the
analytic and skeletal graph Laplacians and his Poincar\'e-Lelong formula
for smooth strict analytic curves \cite[Th\'eor\`eme~3.2.10 and
Proposition~3.3.15]{Thuillier}. For metrized analytic line bundles,
Chambert-Loir and Ducros give a coordinate-free analytic
Poincar\'e-Lelong current identity \cite[Theorem~18.1.6]{CLD}.
Piecewise-linear metrics are formal metrics
\cite[Proposition~3.8]{Vilsmeier}.

\subsection{Analytic divisors, line bundles, finite morphisms, and degrees}\label{subsec:analytic-degree-background}

The analytic $G$-topology is the Grothendieck topology whose
objects are analytic domains and whose coverings are admissible analytic
coverings.  An analytic line bundle on $\cX$ is a locally free sheaf
of rank one on this topology.  A local frame is a nowhere-vanishing local
section that trivializes the line bundle.  An analytic Cartier divisor $A$ on $\cX$ determines
an analytic line bundle $\mathcal O_{\cX}(A)$.  If $A$ is effective,
it has a canonical analytic section $s_A$ satisfying
\[
        \operatorname{div}(s_A)=A.
\]
For a finite analytic morphism
$g:\cX\to \cY$ between proper connected quasi-smooth
strictly $K$-analytic Berkovich curves, write $d_x(g)$ for its analytic
local degree at $x\in\cX$. At a nonclassical point,
$d_x(g)=[\mathcal H(x):\mathcal H(g(x))]$; at a classical point, it is
the ramification index of the local discrete valuation rings. More precisely,
if $\mathfrak m_x$ and $\mathfrak m_{g(x)}$ are their maximal ideals, then
\[
 \mathfrak m_{g(x)}\mathcal O_{\cX,x}
 =\mathfrak m_x^{d_x(g)}.
\]
See \cite[\S3.2.1, Remark~3.2.2(iii), and \S3.2.5]{CTT}.
Its analytic degree is
\[
        \deg g:=\sum_{x\in g^{-1}(y)}d_x(g),
\]
which is independent of $y\in\cY$
\cite[Remark~3.2.2(iii)]{CTT}.

On a quasi-smooth analytic curve over the algebraically closed field $K$, the
Cartier divisors used here, including divisors of nonzero meromorphic
sections, are locally finite integer sums of $K$-rational type-I points.
Properness makes their support finite.  Thus we write
\[
        A=\sum_x n_x[x]
\]
with $x\in\cX$ and set
\[
        \deg A:=\sum_x n_x.
\]
Because $K$ is algebraically closed, these points have residue degree
one.  This is the standard divisor formalism for smooth proper analytic
curves; see \cite[\S\S1.7, 4.3]{Lutkebohmert}.  We use the standard
degree homomorphism on analytic line bundles over a proper analytic curve,
normalized by
\[
        \deg\mathcal O_{\cX}(A)=\deg A.
\]

\begin{lemma}[Degree formulas for analytic curves]\label{lem:analytic-degree-facts}
See \cite[\S\S1.7, 4.3 and Theorem~1.8.1(c)]{Lutkebohmert}
and \cite[Remark~3.2.2(iii)]{CTT}.
Let $\cX, \cY$ be two proper connected quasi-smooth strictly $K$-analytic
Berkovich curves.
Write
\[
        g(\cX):=\dim_K H^1(\cX,\mathcal O_{\cX})
\]
for its analytic genus.
\begin{enumerate}[label=\textup{(\roman*)},leftmargin=2.5em]
\item For every analytic Cartier divisor $A$ on $\cX$,
\[
        \deg\mathcal O_{\cX}(A)=\deg A.
\]
Consequently, if $M$ is an analytic line bundle on $\cX$ and $s$ is a
nonzero meromorphic section of $M$, then
\[
        \deg\operatorname{div}(s)=\deg M.
\]
\item Let
$g:\cX \to\cY$ be finite of analytic degree $d$, and let
$M$ be an analytic line bundle on $\cY$.  Then
\[
        \deg(g^*M)=d\cdot\deg M.
\]
\item For the canonical bundle
\[
        K_{\cX}:=\Omega^1_{\cX/K},
\]
one has
\[
        \deg K_{\cX}=2g(\cX)-2,
\] where $g(\cX)$ denotes the genus of $\cX.$
\end{enumerate}
\end{lemma}

\begin{remark}[Finiteness]\label{rem:finiteness}
A morphism \(f:\cX\to\cY\) of strictly \(K\)-analytic curves is
\emph{finite} if it is topologically proper with finite fibres and the
induced maps on local rings make \(\cO_{\cX,x}\) a finitely generated
\(\cO_{\cY,f(x)}\)-module. For the nonconstant finite morphisms between
proper connected curves considered here, the local degrees satisfy
\(\sum_{x\in f^{-1}(y)}d_x(f)=\deg f\) independently of \(y\), as recalled
above. In this setting, finiteness ensures that the pullback \(f^*D\) of an
analytic Cartier divisor is a finite divisor, so that the counting functions
\(N_f(D,R)\) and \(N_{f^*D}(R)\) are well defined, and makes the
degree \(\deg f\) available in the characteristic growth
\(T_f(D,R)=\deg(f)\deg(D)\,R+O(1)\). It is the only hypothesis on \(f\) needed
for the first main theorem. The second main theorem requires in addition
that \(f\) be generically \'etale; this condition is explained where it is
first used, in \S\ref{sec:log-jac}.
\end{remark}

\subsection{Tate algebras, Gauss norms, and Gauss points}\label{subsec:tate-gauss}

\begin{definition}[One-variable Tate algebra]\label{def:tate-algebra}\cite[Chapter~5]{BGR}.
For $r>0$, set
\[
 K\{r^{-1}T\}
 :=
 \left\{
 F(T)=\sum_{n=0}^{\infty}a_nT^n:
 |a_n|r^n\longrightarrow0
 \right\}.
\]
Its Gauss norm is
\begin{equation}\label{eq:gauss-norm}
        \|F\|_r:=\max_{n\geq0}|a_n|r^n.
\end{equation}
\end{definition}

The condition $|a_n|r^n\to0$ has two immediate consequences. If
$|z|\leq r$, then
\[
        |a_nz^n|\leq |a_n|r^n\longrightarrow0,
\]
so the series converges at every point of the closed disk $B(0,r)$. Moreover,
the sequence $|a_n|r^n$ tends to zero, and its supremum is attained among
finitely many terms. Thus the maximum in \eqref{eq:gauss-norm} is attained.

\begin{definition}[Banach $K$-algebra]\label{def:banach-algebra}
\cite[Chapter~2]{BGR}.
A Banach $K$-algebra is a commutative unital
$K$-algebra $A$ equipped with a norm
\[
        \|\cdot\|_A:A\longrightarrow\R_{\geq0}
\]
such that
\[
\begin{gathered}
 \|f\|_A=0\Longleftrightarrow f=0,\qquad
 \|f+g\|_A\leq\max\{\|f\|_A,\|g\|_A\},\\
 \|cf\|_A=|c|\|f\|_A,\qquad
 \|fg\|_A\leq\|f\|_A\|g\|_A,\qquad
 \|1\|_A=1,
\end{gathered}
\]
and such that $A$ is complete for $d(f,g)=\|f-g\|_A$.
\end{definition}

\begin{definition}[Berkovich spectrum and Berkovich point]\label{def:berkovich-spectrum}
Let $A$
be a commutative Banach $K$-algebra. Its Berkovich spectrum
$\mathcal M(A)$ is the set of bounded multiplicative seminorms
\[
 x:A\longrightarrow\R_{\geq 0},\qquad f\longmapsto |f(x)|,
\]
whose restriction to $K$ is the fixed absolute value. Thus, for
$f,g\in A$ and $c\in K$,
\[
\begin{gathered}
 |0(x)|=0,\quad |1(x)|=1,\quad |(fg)(x)|=|f(x)|\,|g(x)|,\\
 |(f+g)(x)|\leq\max\{|f(x)|,|g(x)|\},\quad |c(x)|=|c|.
\end{gathered}
\]
Boundedness means that $|f(x)|\leq C_x\|f\|_A$ for some $C_x>0$ and all
$f\in A$. A member of $\mathcal M(A)$ is called a Berkovich point. The
topology on $\mathcal M(A)$ is the weakest topology for which
$x\mapsto |f(x)|$ is continuous for every $f\in A$. (See \cite[Chapter~1]{Berkovich} and
\cite[Section~3]{TemkinIntro}).
\end{definition}

\begin{definition}[Closed Berkovich disk]\label{def:berkovich-closed-disk}
\cite[Chapter~1, \S2]{Berkovich}.
For $a\in K$ and $r>0$, the closed Berkovich disk with centre $a$ and radius
$r$ is the affinoid space
\[
 \mathbb B(a,r):=\mathcal M\bigl(K\{r^{-1}(T-a)\}\bigr).
\]
Its classical points are the evaluation seminorms at elements
$z\in K$ satisfying $|z-a|\leq r$.
\end{definition}

It is known from \cite[Chapter~5]{BGR} that the Tate $K\{r^{-1}T\}$ is complete for the Gauss norm, and the Gauss norm is multiplicative:
\begin{equation}\label{eq:gauss-multiplicative}
        \|FG\|_r=\|F\|_r\|G\|_r.
\end{equation}
\begin{proposition}[Maximum-modulus formula for a closed disk]\label{prop:maximum-modulus-disc}
See \cite[Chapter~5, \S1]{BGR} and
\cite[Chapter~1, \S2]{Berkovich}.
For $r>0$ and
$F\in K\{r^{-1}T\}$,
\begin{equation}\label{eq:maximum-modulus-disc}
        \sup_{|z|\leq r}|F(z)|=\|F\|_r.
\end{equation}
\end{proposition}

\begin{definition}[Gauss point]\label{def:gauss-point}
See \cite[Chapter~1, \S2]{Berkovich} and
\cite[Appendix~B]{BakerRumelyBook}.
The Gauss point of the closed disk of
radius $r$ centred at $0$ is denoted by $\zeta_{0,r}$. It is the evaluation
rule on the entire Tate algebra given by
\begin{equation}\label{eq:gauss-point-evaluation}
 |F(\zeta_{0,r})|
 :=\|F\|_r
 =\max_{n\geq0}|a_n|r^n,
 \qquad F=\sum_{n\geq0}a_nT^n.
\end{equation}
More generally, the Gauss-type point of $B(a,r)$ is denoted by
$\zeta_{a,r}$. If $F(T)=\sum_{n\geq0}c_n(T-a)^n$, then
\[
        |F(\zeta_{a,r})|=\max_n|c_n|r^n.
\]
\end{definition}

Equations~\eqref{eq:gauss-multiplicative} and
\eqref{eq:ultrametric-inequality} show
that $F\mapsto|F(\zeta_{0,r})|$ is a bounded multiplicative
non-Archimedean seminorm. Hence it is a Berkovich point in the sense of
Definition~\ref{def:berkovich-spectrum}. A classical point assigns
$F\mapsto|F(a)|$, while the Gauss point records the maximum modulus of every
$F$ on the whole closed disk:
\begin{equation}\label{eq:gauss-records-disc}
        |F(\zeta_{0,r})|=\sup_{|z|\leq r}|F(z)|.
\end{equation}
As $r$ varies, the points $\zeta_{0,r}$ form the Gauss path. For fixed $F$,
\begin{equation}\label{eq:radial-as-path}
        \log\|F\|_r=\log|F(\zeta_{0,r})|.
\end{equation}
The classical radial growth function is therefore the restriction of the
Berkovich function $x\mapsto\log|F(x)|$ to the Gauss path.

\subsection{The Berkovich affine line and its four types of points}\label{subsec:berkovich-spectrum}\label{subsec:four-types}

\begin{definition}[Standard Berkovich affine and projective lines]\label{def:analytification-notation}
\cite[Chapter~3]{Berkovich}.
We write $\mathbf A_K^{1,\an}$ and $\mathbf P_K^{1,\an}$ for the
Berkovich analytic affine and projective lines over $K$.
\end{definition}

For $A=K\{r^{-1}T\}$, each $a\in K$ with $|a|\leq r$ defines the
classical point $x_a:F\mapsto|F(a)|$. The Gauss point $\zeta_{0,r}$ also
belongs to $\mathcal M(A)$. Their kernels distinguish them precisely: the
kernel of the classical evaluation point $x_a$ is the ideal $(T-a)$,
whereas $T-a$ does not belong to the kernel of the Gauss point, since
\[
        |(T-a)(\zeta_{0,r})|=\max\{r,|a|\}=r
        \qquad(|a|\leq r).
\]
\begin{definition}[The four types of points on the Berkovich affine line]\label{def:four-types}
See \cite[Chapter~1]{Berkovich},
\cite[Appendix~B]{BakerRumelyBook}, and \cite[Section~2]{TemkinIntro}.
Assume that $K$ is complete and algebraically
closed. Points of $\mathbf A_K^{1,\an}$ have the following four types.
\begin{enumerate}[label=\textup{(\roman*)}]
\item Type~I points are $\zeta_{a,0}$ with $a\in K$; these are the classical
points.
\item Type~II points are $\zeta_{a,r}$ with $r\in|K^\times|$ and $r>0$.
The unit Gauss point $\zeta_{0,1}$ is the basic example.
\item Type~III points are $\zeta_{a,r}$ with $r>0$ and
$r\notin|K^\times|$.
\item Type~IV points are represented by strictly decreasing families of
closed disks with empty intersection.
\end{enumerate}
\end{definition}

\subsection{Completed residue fields, residue curves, and directional slopes}\label{subsec:completed-residue}

Let $A$ be a $K$-affinoid algebra and $x\in\mathcal M(A)$.  The seminorm
kernel
\[
        \mathfrak p_x:=\{f\in A:|f(x)|=0\}
\]
is prime.  The seminorm induces an absolute value on
$\kappa(x):=\operatorname{Frac}(A/\mathfrak p_x)$.

\begin{definition}[Completed residue field]\label{def:completed-residue}
See \cite[Chapter~1, \S2]{Berkovich} and
\cite[Section~3]{TemkinIntro}.
The \emph{completed residue field} at $x$ is
\[
        \mathcal H(x):=\widehat{\kappa(x)}.
\]
If $\iota_x:A\to\mathcal H(x)$ is the canonical map, then
$f(x):=\iota_x(f)$ and the Berkovich seminorm is $|f(x)|=|\iota_x(f)|$.
\end{definition}

For any complete valued field $L$, put
\[
 L^\circ:=\{u\in L:|u|\leq1\},\qquad
 L^{\circ\circ}:=\{u\in L:|u|<1\},\qquad
 \widetilde L:=L^\circ/L^{\circ\circ}.
\]
The quotient map $L^\circ\to\widetilde L$ is the reduction map, and the
image of $u\in L^\circ$ is denoted by $\widetilde u$.  In particular,
\[
        \widetilde{\mathcal H}_x
        :=\mathcal H(x)^\circ/\mathcal H(x)^{\circ\circ}.
\]

Let $x$ be a type~II point of the quasi-smooth curve $\cX$.  Then
\[
 |\mathcal H(x)^\times|=|K^\times|,
 \qquad
 \operatorname{trdeg}_{\widetilde K}\widetilde{\mathcal H}_x=1;
\]
see \cite[\S2]{TemkinIntro} and \cite[\S5.13]{BPR}.

\begin{definition}[Residue curve]\label{def:residue-curve}
See \cite[Chapter~I, \S6]{Hartshorne} and
\cite[Chapter~7]{Liu}.
The \emph{residue curve} $C_x$ of a type~II point $x$ is the smooth
projective integral curve over $\widetilde K$ characterized, up to
$\widetilde K$-isomorphism, by
\begin{equation}\label{eq:residue-curve-function-field}
        \widetilde K(C_x)\cong\widetilde{\mathcal H}_x.
\end{equation}
Its existence and uniqueness follow from the correspondence between
one-variable function fields and smooth projective curves.
\end{definition}

\begin{definition}[Tangent direction]\label{def:tangent-direction}
See \cite[Lemma~5.12]{BPR} and
\cite[Chapter~4]{Berkovich}.
Let $x\in\cX$.  A tangent direction at $x$ is a germ, under restriction to
smaller connected neighborhoods $V$ of $x$, of a connected component of
$V\setminus\{x\}$.  The set of tangent directions is denoted by $T_x\cX$.
\end{definition}

At a type~II point, reduction gives a natural bijection
\begin{equation}\label{eq:directions-closed-points}
        T_x\cX\longleftrightarrow
        \{\text{closed points of }C_x\},
        \qquad v\longleftrightarrow p_v;
\end{equation}
see \cite[Lemma~5.12 and \S5.13]{BPR}.  This correspondence translates a
directional analytic slope into the order of a rational function on $C_x$.

Let $\mathscr M(\cX)$ be the field of global meromorphic functions on
$\cX$, let $g\in\mathscr M(\cX)^\times$, and let $x$ be type~II.  Choose
$c\in K^\times$ such that $|c|=|g(x)|^{-1}$.  Then $|cg(x)|=1$, and its
reduction defines
\[
        \widetilde g_x:=\widetilde{cg(x)}
        \in\widetilde{\mathcal H}_x^\times
        \cong\widetilde K(C_x)^\times.
\]
Changing $c$ multiplies $\widetilde g_x$ by an element of
$\widetilde K^\times$; hence all orders of $\widetilde g_x$ are independent
of this choice.

For a closed point $p\in C_x$, the local ring $\mathcal O_{C_x,p}$ is a
discrete valuation ring.  A local parameter at $p$ is a generator $u$ of its
maximal ideal.  Every $h\in\widetilde K(C_x)^\times$ has a unique expression
\[
        h=u^m\eta,
        \qquad m\in\mathbf Z,
        \qquad \eta\in\mathcal O_{C_x,p}^\times,
\]
and $\ord_p(h):=m$; see \cite[Chapter~I, \S6]{Hartshorne} or
\cite[Chapter~7]{Liu}.

Let $F$ be real-valued and piecewise affine on the annular germ representing
$v\in T_x\cX$.  If $\gamma_v:[0,\varepsilon)\to\cX$ is parametrized by
arclength, satisfies $\gamma_v(0)=x$, and represents $v$, define
\[
        \partial_vF(x):=
        \lim_{t\rightarrow 0^+}
        \frac{F(\gamma_v(t))-F(x)}{t}.
\]

\begin{lemma}[Reduction-directional-slope formula]\label{lem:reduction-slope}
See \cite[Proposition~3.3.15]{Thuillier} and
\cite[Lemma~5.12 and \S5.13]{BPR}.
Let $x\in\cX$ be type~II, let $v\in T_x\cX$ correspond to
$p_v\in C_x$, and let $g\in\mathscr M(\cX)^\times$.  With arclength away
from $x$ taken as the positive coordinate,
\begin{equation}\label{eq:reduction-directional-slope}
        \partial_v(-\log|g|)(x)
        =\ord_{p_v}(\widetilde g_x).
\end{equation}
\end{lemma}

\subsection{Disks, annuli, and their canonical skeleta}\label{subsec:disks-annuli}

Disks and annuli are the local building blocks of smooth Berkovich curves;
see \cite[Chapter~5]{BGR}, \cite[Appendix~B]{BakerRumelyBook}, and
\cite[Sections~2 and~5]{BPR}.

For disk paths and logarithmic distance, see
\cite[Appendix~B]{BakerRumelyBook} and \cite[Section~5]{BPR}. The closed disk
$B(a,r)$ determines the point $\zeta_{a,r}$. If $0<r_1\leq r_2$, then
\[
        [\zeta_{a,r_1},\zeta_{a,r_2}]
        =\{\zeta_{a,\rho}:r_1\leq\rho\leq r_2\}.
\]
The path metric is normalized by
\begin{equation}\label{eq:logarithmic-distance}
        d(\zeta_{a,r_1},\zeta_{a,r_2})
        =|\log r_2-\log r_1|.
\end{equation}
The classical point $a=\zeta_{a,0}$ lies at an end of this tree. For fixed
$r_0>0$ and $0<\rho\leq r_0$, the distance
$d(\zeta_{a,r_0},\zeta_{a,\rho})=\log(r_0/\rho)$ tends to infinity as
$\rho\to0$.

For Laurent expansions, canonical skeleta, and retractions of standard
annuli, see \cite[Chapter~5]{BGR} and \cite[Section~2]{BPR}. Let $0<r<R$.
The standard closed annulus is
\[
        A[r,R]=\{x:r\leq|T(x)|\leq R\}.
\]
Its canonical skeleton is
\begin{equation}\label{eq:annulus-skeleton}
        \Sigma(A[r,R])
        =\{\zeta_{0,\rho}:r\leq\rho\leq R\}.
\end{equation}
With the coordinate $t=-\log\rho$, this skeleton is a closed interval of
length $\log(R/r)$.  Each connected component of the complement of the
skeleton is an attached open disk.  A point in such a component is joined to
the skeleton by a unique embedded arc, whose endpoint is the Gauss-type point
with the same value of $|T|$.  Collapsing every such component to its
attachment point gives the continuous retraction
\[
        \tau_A:A[r,R]\longrightarrow\Sigma(A[r,R]).
\]
An analytic function on the annulus has a Laurent expansion
$F(T)=\sum_{n\in\mathbf Z}a_nT^n$ satisfying
\[
        |a_n|R^n\longrightarrow0\quad(n\to+\infty),
        \qquad
        |a_n|r^n\longrightarrow0\quad(n\to-\infty).
\]
Along the skeleton,
\begin{equation}\label{eq:annulus-gauss}
 |F(\zeta_{0,e^{-t}})|
 =\max_{n\in\mathbf Z}|a_n|e^{-nt},
\end{equation}
and hence
\begin{equation}\label{eq:annulus-newton}
 -\log|F(\zeta_{0,e^{-t}})|
 =\min_{n\in\mathbf Z}\{\val(a_n)+nt\}.
\end{equation}
The right-hand side is the lower envelope of affine functions of integral
slope $n$. It is therefore continuous and piecewise integral-affine, and its
breakpoints are precisely the points where the dominant exponent changes.

\subsection{Semistable pairs, generalized skeleta, and retractions}\label{subsec:semistable-skeleton}

\begin{definition}[Semistable vertex set]\label{def:semistable-vertex-set}\cite{Thuillier,CTT}.
Let $\cX$ be a connected proper quasi-smooth strictly $K$-analytic Berkovich
curve. A finite set $V\subset\cX$ of type~II points is a semistable vertex
set if every connected component of $\cX\setminus V$ is an open disk or an
open annulus and only finitely many components are annuli. The union of $V$
with the canonical skeleton of each annular component is the finite skeleton
$\Sigma(\cX,V)$.
\end{definition}

Semistable triangulations exist for proper quasi-smooth strictly
$K$-analytic Berkovich curves.  We use the analytic semistable
triangulation and refinement theory summarized in \cite[\S3.5]{CTT}: after
refinement, any finite collection of prescribed type~II points needed below
may be included in the vertex set.  We work over the fixed algebraically
closed field $K$.

Let $H\subset\cX$ be a nonempty finite set, also viewed as the reduced
effective Cartier divisor $\sum_{h\in H}[h]$. After a semistable refinement,
each $h$ is approached through a punctured open disk attached to the compact
skeleton. Deleting $h$ turns the geodesic toward that type~I endpoint into an
unbounded ray, which we denote by $\rho_h$. Adjoining these rays to the compact
skeleton gives a generalized skeleton of
\[
        U=\cX\setminus H:
        \qquad
        \Sigma=\Sigma(\cX,V,H)\subset U.
\]
There is a strong deformation retraction
\begin{equation}\label{eq:retraction}
        \tau:U\longrightarrow\Sigma.
\end{equation}
Every open-disk branch outside the skeleton retracts to its attachment
point, and puncture directions correspond to unbounded rays;
see \cite{Thuillier,CTT}.

Below, we use the standard pushforward of finite signed zero-cycles and of finite
signed Borel measures. The geometric instance of the divisor pushforward under
skeletal retractions is the one used in \cite{BakerRabinoff2015}; the
measure-theoretic pushforward is the usual one.

\begin{definition}\label{def:skeletal-pushforward}
Let \(q:S\to\Gamma\) be a continuous map from a Berkovich curve or a metric
graph to a metric graph.  If
\[
 Z=\sum_{x\in S}m_x[x]
\]
is a finite signed zero-cycle, define its pushforward as the atomic measure
\(q_*Z\) characterized by
\[
 \int_\Gamma \varphi\,d(q_*Z)
 =\sum_{x\in S}m_x\,\varphi(q(x))
\]
for every continuous real-valued function \(\varphi\) on \(\Gamma\).
For a finite signed Borel measure \(\mu\) on \(S\), its pushforward \(q_*\mu\)
is characterized by
\[
 \int_\Gamma\varphi\,d(q_*\mu)
 =\int_S(\varphi\circ q)\,d\mu
\]
for every bounded continuous real-valued function \(\varphi\) on \(\Gamma\).
\end{definition}

\subsubsection{Common semistable refinements}\label{subsec:refinement-independence}

A semistable refinement replaces a semistable vertex set by a larger one;
the analytic curve is unchanged, while its skeleton is enlarged or
subdivided.  Suppose that finitely many semistable skeletal
representations have already been chosen, for example to represent finitely
many PL (piecewise-linear) metrics.  Their finitely many type-II vertex sets are contained in a
common semistable refinement by the analytic semistable refinement theory
recalled in \cite[\S3.5]{CTT}.  Divisors are treated separately: their type-I
support points do not become vertices of the compact semistable skeleton;
their coefficients are recorded at their retraction points, which may be
inserted as ordinary metric-graph subdivision points.

\subsection{Metric-graph calculus on a fixed skeleton}\label{sec:graph-potential}

The logarithmic potentials of analytic functions and line-bundle metrics
restrict to continuous piecewise-affine functions on adapted skeleta.
Their outgoing slopes define the graph Laplacian, which records divisors
and curvature measures through the slope formula. For metric-graph
potential theory, see \cite{BakerRumelyGraphs,BakerFaberGraphs,BakerRumelyBook}.
We use the sum of outgoing slopes, the negative of the convention in
\cite[Sections~3 and 5]{BakerFaberGraphs}.

\subsubsection{Piecewise-affine functions, outgoing slopes, and the graph Laplacian}
\begin{definition}[Compact metric graph]\label{def:metric-graph}
See \cite[Section~2]{BakerFaberGraphs} and
\cite{BakerRumelyGraphs}.
A compact
metric graph $\Gamma$ is a compact connected metric space obtained by gluing
finitely many closed intervals of positive length along their endpoints. An
interior point of an edge has two tangent directions, while the number of
tangent directions at a vertex is its valence. Inserting two-valent vertices
in edge interiors does not change the metric graph.
\end{definition}

\begin{definition}[Piecewise-affine function]\label{def:pa-function}
See \cite[Sections~2--3]{BakerFaberGraphs} and
\cite{BakerRumelyGraphs}.
A continuous function $F:\Gamma\to\R$ is
piecewise affine if it is affine in the arclength coordinate on every edge of
some finite subdivision. It is piecewise integral-affine if all slopes are
integers.
\end{definition}

\begin{definition}[Directional derivative on a metric graph]\label{def:graph-directional-derivative}
\cite[Sections~2-3]{BakerFaberGraphs}.
Let $x\in\Gamma$, and let $T_x\Gamma$ denote the finite set of germs of
nonconstant geodesic segments issuing from $x$. For $v\in T_x\Gamma$, choose
the unit-speed representative
\[
 \gamma_v:[0,\varepsilon)\longrightarrow\Gamma,
 \qquad \gamma_v(0)=x,
\]
whose germ is $v$. For a function $F$ affine on this germ, its outgoing
directional derivative at $x$ is
\[
        \partial_vF(x)
        :=\left.\frac{d}{dt}F(\gamma_v(t))\right|_{t=0+}.
\]
\end{definition}

\begin{definition}[Dirac measure]\label{def:dirac-measure}
\cite[Section~2]{BakerFaberGraphs}.
Let $x\in\Gamma$.  The Dirac measure at $x$ is the Borel measure
$\delta_x$ characterized by
\[
        \int_\Gamma \varphi\,d\delta_x=\varphi(x)
\]
for every continuous function $\varphi$ on $\Gamma$; equivalently,
$\delta_x(E)=1$ if $x\in E$ and $0$ otherwise for Borel sets
$E\subseteq\Gamma$.
\end{definition}

\begin{definition}[Graph Laplacian]\label{def:graph-lap}
See \cite[Section~3]{BakerFaberGraphs} and
\cite{BakerRumelyGraphs}.
The graph
Laplacian of a piecewise-affine function is the atomic measure
\begin{equation}\label{eq:graph-laplacian-local}
        \Delta_\Gamma F
        :=\sum_{x\in\Gamma}
        \left(\sum_{v\in T_x\Gamma}\partial_vF(x)\right)\delta_x.
\end{equation}
Only vertices and slope breakpoints can have nonzero coefficients.
\end{definition}

At an interior point of an edge, if the derivatives from left to right are
$m_-$ and $m_+$, the left outgoing slope is $-m_-$ and the right outgoing
slope is $m_+$.  Hence the coefficient of the graph Laplacian at that point is
\begin{equation}\label{eq:slope-break-coefficient}
        \Delta_\Gamma F(x)=m_+-m_-.
\end{equation}

\begin{definition}[Harmonic function]\label{def:graph-harmonic}
See \cite[Section~3]{BakerFaberGraphs} and
\cite{BakerRumelyGraphs}.
Let $\Omega\subseteq\Gamma$ be open. A continuous piecewise-affine function
$F$ on $\Gamma$ is harmonic on $\Omega$ if
\[
 \Delta_\Gamma F|_\Omega=0.
\]
Equivalently, $F$ is affine on every edge germ contained in $\Omega$ and
\[
 \sum_{v\in T_x\Gamma}\partial_vF(x)=0
 \qquad(x\in\Omega).
\]
At a boundary point of a compact subgraph, harmonicity is imposed only when
that point belongs to the specified open domain; Dirichlet boundary values
are prescribed separately.
\end{definition}

\subsubsection{Ordinary metric-graph subdivisions}\label{subsec:ordinary-subdivision}

An ordinary metric-graph subdivision inserts finitely many two-valent points
into skeletal edges without changing the underlying metric space or
retraction, and hence without changing divisor pushforwards. At a new
vertex on an affine edge, the two outgoing slopes sum to zero. Thus the
graph Laplacian, skeletal Chern measures, and Green pairings are unchanged.
Only the graph presentation is refined; the semistable vertex set is not altered.

\subsubsection{Integration by parts and Poisson theory}

\begin{lemma}[Total mass of a graph Laplacian]\label{lem:lap-total-zero}
See \cite[Section~3]{BakerFaberGraphs} and
\cite{BakerRumelyGraphs}.
For every compact metric graph $\Gamma$ and every piecewise-affine function
$F$ on $\Gamma$,
\[
        \int_\Gamma d(\Delta_\Gamma F)=0.
\]
\end{lemma}

For functions $F,G$ affine on every edge of a common subdivision, define
\begin{equation}\label{eq:dirichlet-energy}
        \mathcal E(F,G):=\sum_e\int_eF'(t)G'(t)\,dt.
\end{equation}

\begin{lemma}[Integration by parts on a metric graph]\label{lem:graph-ibp-full}
See \cite[Section~3]{BakerFaberGraphs} and
\cite{BakerRumelyGraphs}.
If $\Gamma$ is a compact metric graph and $F,G$ are continuous
piecewise-affine functions on $\Gamma$, then
\begin{equation}\label{eq:early-graph-ibp}
        \int_\Gamma G\,d(\Delta_\Gamma F)
        =-\mathcal E(F,G)
        =\int_\Gamma F\,d(\Delta_\Gamma G).
\end{equation}
\end{lemma}

\begin{lemma}[Maximum principle]\label{lem:max-principle}
\cite{BakerRumelyGraphs,BakerRumelyBook}.
If $\Gamma$ is a compact connected metric graph and $F$ is a continuous
piecewise-affine function on $\Gamma$ with $\Delta_\Gamma F=0$, then $F$
is constant.
\end{lemma}

\begin{lemma}[Poisson equation on a compact metric graph]\label{lem:poisson-compact}
\cite{BakerRumelyGraphs,BakerFaberGraphs}.
Let $\mu$ be a finite atomic
measure on a compact connected metric graph $\Gamma$. The equation
\[
        \Delta_\Gamma F=\mu
\]
has a continuous piecewise-affine solution if and only if
$\mu(\Gamma)=0$. The solution is unique up to an additive constant.
\end{lemma}

\subsection{PL/formal metrics on analytic line bundles}\label{subsec:formal-metrics}

Let $L$ be an analytic line bundle on $\cX$.  A continuous metric on $L$
assigns to each local analytic section a continuous norm and satisfies
\[
        \|(hs)(x)\|_x=|h(x)|\,\|s(x)\|_x
\]
for every analytic domain $V\subset\cX$, every
$h\in\mathcal O_{\cX}(V)$, every $s\in\Gamma(V,L)$, and every $x\in V$;
see \cite[Definition~3.1]{Vilsmeier}.

If $e$ is a local frame, its metric weight is
\[
        \phi_e:=-\log\|e\|.
\]
We use the standard notion of a piecewise-linear (PL) metric from
\cite[Section~3]{Vilsmeier}.  On a curve, after passage to an adapted skeleton
and a finite subdivision, its local weights restrict to continuous
piecewise-affine functions.  A formal metric is a metric induced by a line
bundle on an admissible formal model of the analytic curve.  On the spaces
considered here, PL metrics are formal and are stable under tensor products
and analytic pullbacks \cite[Propositions~3.7--3.8]{Vilsmeier}.  Semistable
cofinality means that, after an admissible formal blow-up, a chosen formal
model is dominated by a semistable formal model without changing its generic
analytic fibre \cite[Exercise~6.3.8(i)]{TemkinIntro}.  Thus a PL metric may
be represented on a semistable skeleton.

For finitely many metrics, the common semistable refinement of
Subsubsection~\ref{subsec:refinement-independence} provides a fixed skeleton
adapted to all of them. Subsequent insertions are ordinary metric-graph
subdivisions.

A nonzero meromorphic section $s$ of $L$ has an analytic Cartier divisor
$\operatorname{div}(s)$.  The quantity $-\log\|s\|$ may take an infinite
value at a type-I zero or pole of $s$.  All graph-theoretic expressions below
are therefore taken on skeleta, which contain no such divisor points
themselves.  For any skeleton or finite skeletal subgraph $\Lambda$ under
consideration, write
\[
        U_s|_\Lambda:=(-\log\|s\|)|_\Lambda,
\]
and abbreviate this restriction to $U_s$ when the graph is clear.  On these
graphs $U_s$ is a finite continuous piecewise-affine function after the
stated subdivision.

\subsection{Slope formulas and skeletal first Chern measures}\label{subsec:slope-pl}

The graph Laplacian and its sign convention have already been fixed in
Definition~\ref{def:graph-lap}: it is the sum of outgoing slopes.  In
particular, equations~\eqref{eq:graph-laplacian-local} and
\eqref{eq:slope-break-coefficient} apply to every adapted skeleton below.

\begin{lemma}[Skeletal slope formula]\label{lem:slope-formula}
See \cite[Th\'eor\`eme~3.2.10, \S3.3, and
Proposition~3.3.15]{Thuillier}.
Let $H\subset\cX$ be a finite puncture set, put $U:=\cX\setminus H$, and
let $\tau:U\to\Sigma$ be a generalized-skeleton retraction.  Let
$g\in\mathscr M(\cX)^\times$, and assume that $\Sigma$ is adapted to the
retractions of the zeros and poles of $g$ in $U$ and to the breakpoints of
$-\log|g|$. Then $F_g:=(-\log|g|)|_\Sigma$ is continuous and piecewise
integral-affine, and
\begin{equation}\label{eq:slope-generalized}
        \Delta_\Sigma(-\log|g|)
        =-\tau_*\bigl(\operatorname{div}(g)|_U\bigr).
\end{equation}
If $\Gamma\subset\cX$ is a compact semistable skeleton adapted to the
retraction of $\operatorname{div}(g)$ and to the breakpoints of $-\log|g|$,
with retraction $\tau_\Gamma:\cX\to\Gamma$, then
\begin{equation}\label{eq:slope-compact}
        \Delta_\Gamma(-\log|g|)
        =-\tau_{\Gamma*}\operatorname{div}(g).
\end{equation}
The formula is local: let $V\subset\cX$ be an analytic domain,
let $g\in\mathscr M(V)^\times$, and let $\Lambda$ be an adapted skeleton
of $V$, with retraction $\tau_\Lambda:V\to\Lambda$. If
$\Omega\subset\Lambda$ is a relative open subgraph whose inverse image
$\tau_\Lambda^{-1}(\Omega)$ is contained in the interior of $V$ in $\cX$, then
\begin{equation}\label{eq:slope-local}
 \bigl(\Delta_\Lambda(-\log|g|)\bigr)|_\Omega
 =-\bigl(\tau_{\Lambda*}\operatorname{div}(g)\bigr)|_\Omega.
\end{equation}
\end{lemma}

Thuillier's compatibility between the analytic and graph Laplacians,
together with the analytic Poincar\'e-Lelong formula, gives
\eqref{eq:slope-compact} with the convention
$\Delta_\Gamma F(x)=\sum_{v\in T_x\Gamma}\partial_vF(x)$.
The Laurent expansion on each annular edge gives the continuity and
piecewise integral-affinity of $-\log|g|$. Applying the same local identity
on $U=\cX\setminus H$ gives \eqref{eq:slope-generalized}, and applying this
identity on the interior of $V$ gives \eqref{eq:slope-local}.

\begin{definition}[Skeletal first Chern measure]\label{def:skeletal-chern}
\cite[4.3 and Definition~4.7]{Vilsmeier}.
Let $(L,\|\cdot\|)$ be a PL-metrized analytic line bundle on $\cX$.
Choose a semistable formal $K^\circ$-model $\mathfrak X$, together with
an isomorphism of its generic fibre with $\cX$, and a formal line bundle
$\mathfrak L$, together with an isomorphism of its generic fibre with $L$
under which the induced metric equals $\|\cdot\|$. Write
$\mathfrak X_{\widetilde K}$ for the special fibre. For each irreducible
component $Y_v$, let $v$ be the unique analytic point reducing to its
generic point. These points are the vertices of the associated skeleton;
see \cite[Corollary~2.9(f)]{Vilsmeier}.
The formal first Chern measure is
\[
 c_1(L,\|\cdot\|)
 :=\sum_v\deg(\mathfrak L|_{Y_v})\,\delta_v.
\]
Let $\Lambda$ be either a compact semistable skeleton or a generalized
skeleton in $U=\cX\setminus H$. In this definition, adaptation to the
metric means that $\Lambda$ contains the skeleton of such a formal model.
In particular it contains the support of $c_1(L,\|\cdot\|)$. Define
\begin{equation}\label{eq:skeletal-curvature-local-frame}
 \omega_{L,\Lambda}
 :=\sum_v\deg(\mathfrak L|_{Y_v})\,\delta_v,
\end{equation}
now regarded as a measure on $\Lambda$.
\end{definition}

This measure is independent of the formal model representing the given
metric; see \cite[Definition~4.7 and the discussion following it]{Vilsmeier}.
Ordinary metric-graph subdivision leaves the measure unchanged.
It has finite support and hence is compactly supported also on a
generalized skeleton. This is the curve case of the formal
Chambert-Loir measure \cite{ChambertLoirMeasures,Vilsmeier}; its analytic
Poincar\'e-Lelong interpretation is compatible with
\cite[Theorem~18.1.6]{CLD}.

\subsection{Green kernels on compact metric graphs}\label{subsec:green-kernels}

\begin{definition}[Two-point Green kernel]\label{def:two-point-green}
\cite{BakerRumelyGraphs,BakerFaberGraphs}.
For $a,b\in\Gamma$, the
two-point Green kernel $g_{a,b}$ is the continuous piecewise-affine function
satisfying
\begin{equation}\label{eq:two-point-green}
        \Delta_\Gamma g_{a,b}=\delta_b-\delta_a,
        \qquad g_{a,b}(b)=0.
\end{equation}
\end{definition}

For $a\neq b$, this normalized Green kernel exists uniquely, is nonnegative,
and satisfies $g_{a,b}(a)>0$; see
\cite{BakerRumelyGraphs,BakerFaberGraphs}.

\section{Finite truncations and multi-tail Dirichlet kernels}
\label{sec:multi-tail-kernel}

Retain from Section~\ref{sec:prelim} the proper connected quasi-smooth strictly
$K$-analytic Berkovich curve $\cX$, the nonempty finite set
\(H\subset\cX(K)\), identified with its reduced effective Cartier divisor, and
\[
        U=\cX\setminus H,
\]
together with the generalized skeleton $\Sigma\subset U$ and the retraction
$\tau:U\to\Sigma$. The skeleton is chosen after a common semistable
refinement adapted to the finitely many metrics under consideration,
before fixing the core and tail coordinates below.
Put $n_H:=\#H=\deg H$. After a finite metric-graph
subdivision, choose a finite compact connected subgraph
$\Gamma_0\subset\Sigma$ containing every bounded edge and every branch point,
such that $\Sigma\setminus\Gamma_0$ is the disjoint union of terminal open
rays. Write
\[
        \Sigma=\Gamma_0\cup\bigcup_{h\in H}\rho_h,
\]
where $\rho_h$ is the closed unbounded ray corresponding to $h$. Its
attachment point is the unique point
\[
        v_h\in\Gamma_0\cap\rho_h.
\]
Choose the arclength coordinate and its inverse parametrization
\[
 x_h:\rho_h\longrightarrow[0,\infty),\qquad
 \xi_h:=x_h^{-1}:[0,\infty)\longrightarrow\rho_h,
 \qquad x_h(v_h)=0.
\]
Thus $\xi_h(0)=v_h$, and $x_h$ tends to $+\infty$ toward $h$. Explicitly,
\[
 d(\xi_h(s),\xi_h(t))=|s-t|,
 \qquad x_h(\xi_h(t))=t
 \quad(s,t\geq0).
\]
Fix a base point $b\in\Gamma_0$.

\subsection{Truncations and stabilized tail behavior}

For every $R>0$,   define the finite truncation
\[
 \Sigma_R:=\Gamma_0\cup\bigcup_{h\in H}[v_h,\xi_h(R)]
\]
and its boundary
\[
        \partial\Sigma_R:=\{\xi_h(R):h\in H\}.
\]
The point $\xi_h(R)$ is an auxiliary two-valent vertex of the ambient
metric graph $\Sigma$ and a one-valent boundary vertex of $\Sigma_R$; it
need not be a semistable vertex.  Define the clamping map
\[
 q_R:\Sigma\longrightarrow\Sigma_R,\qquad
 q_R(\xi_h(t))=\xi_h(\min\{t,R\})
\]
on $\rho_h$, and let $q_R$ be the identity on $\Gamma_0$.  The associated
finite-truncation retraction is
\[
        \tau_R:=q_R\circ\tau:U\longrightarrow\Sigma_R.
\]

We apply compact metric-graph potential theory to the finite truncations
$\Sigma_R$ and then study growth and defect relations as $R\to\infty$.

\begin{convention}\label{conv:tail-normalization}
Let $L$ be an analytic line bundle on $\cX$. A piecewise-linear metric on $L$
is \emph{tail-normalized admissible} relative to $(\Sigma,H)$ if the fixed
generalized skeleton $\Sigma$ is adapted to its local metric weights after
the necessary finite metric-graph subdivision and each puncture $h\in H$
has a nowhere-vanishing analytic frame $e_h$ on a neighborhood of $h$ satisfying
\(\lVert e_h\rVert=1\) on a terminal subray of $\rho_h$.
\end{convention}

\begin{lemma}\label{lem:tail-normalized-metric-existence}
Every analytic line bundle on $\cX$ admits a tail-normalized admissible PL
metric relative to a suitable common semistable refinement of $\Sigma$.
\end{lemma}

\begin{proof}
\noindent\textit{Step 1.}
For each $h\in H$, choose a sufficiently small compact strictly analytic disk
$W_h\subset\cX$ on which $L$ is trivial, with the $W_h$ pairwise disjoint,
and choose a frame $e_h$ of $L|_{W_h}$.  The rule $\|e_h\|=1$ defines a
formal metric on $L|_{W_h}$. Since $H$ is finite,
$W:=\coprod_{h\in H}W_h$ is a compact strictly analytic domain.

\smallskip
\noindent\textit{Step 2.}
By \cite[Proposition~3.3]{Vilsmeier}, this formal metric on $L|_W$ extends to a
formal metric on $L$ over $\cX$.

\smallskip
\noindent\textit{Step 3.}
By semistable cofinality \cite[Exercise~6.3.8(i)]{TemkinIntro}, the formal
metric is represented on a semistable skeleton. For finitely many metrics,
the common refinement of Subsubsection~\ref{subsec:refinement-independence}
provides a generalized skeleton adapted to all of them. This refinement may
enlarge the skeleton, rather than merely insert two-valent vertices; the
resulting skeleton is then fixed as above.

\smallskip
\noindent\textit{Step 4.}
Every sufficiently far terminal subray of $\rho_h$ lies in
$W_h\setminus\{h\}$, where $\|e_h\|=1$ by construction.
\end{proof}

\begin{lemma}\label{lem:eventual-tail-potential}
Let $(L,\|\cdot\|)$ be an analytic line bundle with a tail-normalized
admissible PL metric, and let $s$ be a nonzero meromorphic section of $L$.
Define $\ord_h(s)$ to be the
coefficient of $h$ in $\operatorname{div}(s)$.  For every $h\in H$ there are
$T_h\geq0$ and $c_h(s)\in\mathbb R$ such that
\begin{equation}\label{eq:eventual-tail-potential}
 U_s(\xi_h(t))=\ord_h(s)t+c_h(s),
 \qquad U_s=(-\log\|s\|)|_\Sigma,
\end{equation}
for all $t\geq T_h$.
\end{lemma}

\begin{proof}
The quasi-smoothness of $\cX$ implies that the local ring at the type-I point
$h$ is a one-dimensional regular local ring, hence a discrete valuation
ring.  Choose a local parameter $z_h$, meaning a generator of its maximal
ideal.  On a neighborhood of $h$ write
\[
        s=z_h^{\ord_h(s)}u_h e_h,
\]
where $u_h$ is an analytic unit and $e_h$ is the normalized frame from
Convention~\ref{conv:tail-normalization}.  Since $u_h(h)\neq0$, continuity
allows the neighborhood to be reduced until
$|u_h-u_h(h)|<|u_h(h)|$.  Lemma~\ref{lem:unequal-norm} then gives
$|u_h|=|u_h(h)|$ there. After shrinking this neighborhood, $z_h$ identifies
it with a disk centered at $0$, and the terminal ray toward $h$ is the radial
path of Gauss points in this coordinate. Choose $t_0$ so that this description
holds for $t\geq t_0$, and put $r_h(t):=|z_h(\xi_h(t))|$.
Since $\xi_h$ is parametrized by unit arclength toward $h$, the logarithmic
disk distance formula \eqref{eq:logarithmic-distance} gives
\[
 t-t_0=d(\xi_h(t_0),\xi_h(t))
       =\log\frac{r_h(t_0)}{r_h(t)}.
\]
Consequently, with $\kappa_h:=-\log r_h(t_0)-t_0$, we have
\[
 -\log|z_h(\xi_h(t))|=t+\kappa_h
\]
for all $t\geq t_0$.  Since $\|e_h\|=1$ on the same terminal
subray, taking $-\log$ of the preceding factorization gives
\eqref{eq:eventual-tail-potential}.
\end{proof}

\subsection{Skeletal Poincar\'e-Lelong formulas and Chern measures}\label{R0}

The formal-model definition of the Chern measure and the directional
slope formula give the following line-bundle identities.

\begin{theorem}[Skeletal Poincar\'e-Lelong formula]\label{thm:skeletal-pl}
Let $(L,\|\cdot\|)$ be an analytic line bundle on $\cX$ with a
piecewise-linear metric, let $s$ be a nonzero meromorphic section of $L$, and
let $\Gamma$ be a compact semistable skeleton adapted to the metric and to
$\operatorname{div}(s)$, with retraction $\tau_\Gamma:\cX\to\Gamma$.  Then
\begin{equation}\label{eq:skeletal-poincare-lelong-compact}
        \Delta_\Gamma(-\log\|s\|)
        =\omega_{L,\Gamma}-\tau_{\Gamma*}\operatorname{div}(s).
\end{equation}

Suppose, in addition, that the metric is tail-normalized admissible relative
to the fixed pair $(\Sigma,H)$ in the sense of
Convention~\ref{conv:tail-normalization}.  Then $\omega_{L,\Sigma}$ is compactly
supported, and
\begin{equation}\label{eq:skeletal-poincare-lelong}
        \Delta_\Sigma(-\log\|s\|)
        =\omega_{L,\Sigma}
        -\tau_*\bigl(\operatorname{div}(s)|_U\bigr).
\end{equation}
Moreover, for each $h\in H$, the terminal slope satisfies
\begin{equation}\label{eq:terminal-order-slope}
        \partial_h(-\log\|s\|)=\operatorname{ord}_h(s)
\end{equation}
on every sufficiently far part of $\rho_h$, where $\partial_h$ denotes the
derivative with respect to the fixed unit-speed coordinate $x_h$ increasing
toward $h$.
\end{theorem}

\begin{proof}
Write $\Lambda=\Gamma$ or $\Sigma$, let $q_\Lambda$ be its retraction,
and put $W_s:=-\log\|s\|$ away from the type-I zeros and poles of $s$,
with $U_s:=W_s|_\Lambda$. Choose a representing semistable
formal model as in Definition~\ref{def:skeletal-chern}. Model independence
allows semistable refinements along $\Lambda$. Include the type-II
retraction points of the divisor and the slope breakpoints as model
vertices, and then make any further metric-graph subdivisions needed.
All nonzero coefficients occur at these type-II vertices.
On each remaining open edge a
norm-one formal frame reduces the calculation to the annular formula
\eqref{eq:annulus-newton}, so $U_s$ is affine there.

Fix a type-II vertex $v$. Let $Y_v$ be its special-fibre component, let
$C_v$ be the normalization of $Y_v$, and let $\overline L_v$ be the
pullback of $\mathfrak L|_{Y_v}$ to $C_v$. The degree projection formula
for the finite birational normalization gives
$\deg\overline L_v=\deg(\mathfrak L|_{Y_v})$; see
\cite[Chapter~7]{Liu}. Since formal frames have norm one and
$|\mathcal H(v)^\times|=|K^\times|$, choose $c_v\in K^\times$ with
$\|c_vs(v)\|=1$. Reduction in formal frames defines a nonzero rational
section $\widetilde s_v$ of $\overline L_v$. Indeed, write $c_vs=f_i e_i$
in formal frames. If $e_j=a_{ij}e_i$ on an overlap, then
\[
 f_j=a_{ij}^{-1}f_i,
 \qquad
 \widetilde f_j\widetilde e_j
 =\bigl(\widetilde a_{ij}^{-1}\widetilde f_i\bigr)
   \bigl(\widetilde a_{ij}\widetilde e_i\bigr)
 =\widetilde f_i\widetilde e_i.
\]
Here $\widetilde f_i$ is the residue of $f_i(v)$ and $\widetilde e_i$
is the induced frame on $C_v$. Thus the local expressions agree.
Changing $c_v$ multiplies $\widetilde s_v$ by a nonzero residue-field
constant.

For a tangent direction $\eta\in T_v\cX$ corresponding to $p_\eta\in C_v$, the local reduction--slope calculation (Lemma \ref{lem:reduction-slope}
) gives
\[
 \partial_\eta W_s(v)=\ord_{p_\eta}(\widetilde s_v).
\]
Indeed, in a formal frame near $p_\eta$ the metric weight is zero.
The reduction--slope calculation underlying
\eqref{eq:reduction-directional-slope} is local on an annular direction
germ and therefore applies to the local meromorphic coefficient $f_i$;
compare \cite[Theorem~5.15(3) and its proof]{BPR}.
Summing over all directions, using that $\widetilde K$ is algebraically
closed, gives
\[
 \sum_{\eta\in T_v\cX}\partial_\eta W_s(v)
 =\deg\operatorname{div}(\widetilde s_v)
 =\deg\overline L_v
 =\omega_{L,\Lambda}(\{v\}).
\]
Only finitely many summands are nonzero.

Put $O_v:=T_v\cX\setminus T_v\Lambda$. Each $\eta\in O_v$ enters
an open-disk component $B_\eta$ of $\cX\setminus\Gamma$, or of
$U\setminus\Sigma$, attached at $v$ and containing no puncture.
The model metric has a norm-one frame $e$ on $B_\eta$.
In a unit-disk coordinate $z$, write
\[
 s=u\prod_j(z-a_j)^{m_j}e,
 \qquad m_j=\ord_{a_j}(s),
\]
where the product contains all zeros and poles in $B_\eta$, and $u$ is an
analytic unit. Its absolute value is constant on the open disk; see
\cite[proof of Lemma~2.13]{BPR}. For radii
larger than every $|a_j|$, the Gauss formula gives
\[
 \partial_\eta W_s(v)
 =\sum_jm_j=\deg\bigl(\operatorname{div}(s)|_{B_\eta}\bigr);
\]
see also \cite[Proposition~2.10(2)]{BPR}.
Subtracting these omitted directions from the all-direction sum yields
\[
\begin{aligned}
 (\Delta_\Lambda U_s)(\{v\})
 &=\deg\overline L_v
   -\sum_{\eta\in O_v}\deg\bigl(\operatorname{div}(s)|_{B_\eta}\bigr)\\
 &=\omega_{L,\Lambda}(\{v\})
   -\sum_{q_\Lambda(x)=v}\ord_x(s),
\end{aligned}
\]
where $x$ ranges over $\cX(K)$ in the compact case and over $U\cap\cX(K)$
in the punctured case. This proves
\eqref{eq:skeletal-poincare-lelong-compact} and
\eqref{eq:skeletal-poincare-lelong}; the edge-interior calculation follows
from
\cite[Proposition~2.10(1)]{BPR}.

The Chern measure has finite support by
Definition~\ref{def:skeletal-chern}. Finally,
\eqref{eq:eventual-tail-potential} gives
\[
        \frac{d}{dt}(-\log\|s\|)(\xi_h(t))=\operatorname{ord}_h(s)
\]
for all sufficiently large $t$, which is
\eqref{eq:terminal-order-slope}.
\end{proof}

\begin{lemma}[Total mass on a compact skeleton]\label{lem:chern-total-degree-compact}
Let $\cX$ be a proper connected quasi-smooth strictly $K$-analytic
Berkovich curve, let $(L,\|\cdot\|)$ be an analytic line bundle on $\cX$
with a piecewise-linear metric, and let $\Gamma$ be a compact semistable
skeleton on which the metric is represented.  If $\omega_{L,\Gamma}$ is the
skeletal first Chern measure, then
\[
        \omega_{L,\Gamma}(\Gamma)=\deg L.
\]
\end{lemma}

\begin{proof}
Choose a nonzero meromorphic section $s$ of $L$ and, if necessary, subdivide
the fixed compact metric graph at the retraction points of
$\operatorname{div}(s)$ and at the finitely many slope breakpoints.  This does
not change the graph-Laplacian measure.  By
Theorem~\ref{thm:skeletal-pl},
\[
        \Delta_\Gamma(-\log\|s\|)
        =\omega_{L,\Gamma}-\tau_{\Gamma*}\operatorname{div}(s).
\]
Taking total masses, using Lemma~\ref{lem:lap-total-zero}, and then applying
Lemma~\ref{lem:analytic-degree-facts}, we obtain
\[
        \omega_{L,\Gamma}(\Gamma)
        =\deg\operatorname{div}(s)
        =\deg L.
\]
\end{proof}

For a tail-normalized admissible metric, the generalized skeletal Chern
measure is compactly supported by Theorem~\ref{thm:skeletal-pl}.  Moreover,
\begin{equation}\label{eq:chern-total-degree-generalized}
        \omega_{L,\Sigma}(\Sigma)=\deg L.
\end{equation}
Indeed, choose a nonzero meromorphic reference section $s_0$ and take $R$
so large that the supports of $\tau_*(\operatorname{div}(s_0)|_U)$ and
$\omega_{L,\Sigma}$ lie in the interior of $\Sigma_R$, and all deleted
tails lie in the terminal affine regions of Lemma~\ref{lem:eventual-tail-potential}.
Since $\Sigma_R$ is compact,
Lemma~\ref{lem:lap-total-zero} applied to $U_{s_0}|_{\Sigma_R}$ gives total
mass zero.  Restriction from $\Sigma$ to $\Sigma_R$ contributes the endpoint
coefficients $-\operatorname{ord}_h(s_0)$, so
\[
 0
 =\omega_{L,\Sigma}(\Sigma)
  -\deg\bigl(\operatorname{div}(s_0)|_U\bigr)
  -\sum_{h\in H}\operatorname{ord}_h(s_0).
\]
The last two terms add to $\deg\operatorname{div}(s_0)=\deg L$ by
Lemma~\ref{lem:analytic-degree-facts}, proving
\eqref{eq:chern-total-degree-generalized}.

Pass once to a common semistable refinement adapted to the finitely many
divisors, local metric weights, and tail frames used below, and then make a finite
metric-graph subdivision at $b$, at all relevant divisor retraction points,
and at all breakpoints of the relevant piecewise-affine functions.  The
additional subdivision points need not be semistable vertices.  Such a common
threshold exists because only finitely many fixed divisors, sections, metrics,
tail frames, and breakpoints are involved: the relevant divisor retraction sets
are finite, the skeletal Chern measures have compact support by
Theorem~\ref{thm:skeletal-pl}, and the terminal affine and norm-one conditions
hold eventually on each of the finitely many puncture rays.  Taking the maximum
of the resulting finitely many thresholds, choose a single number $R_0>0$ such
that, for every $R\geq R_0$:
\begin{enumerate}[label=\textup{(\roman*)},leftmargin=2.5em]
\item all retraction points of the divisors occurring in the identity lie
in the interior of $\Sigma_R$;
\item the supports of all skeletal Chern measures occurring in the identity
lie in the interior of $\Sigma_R$;
\item the affine formula \eqref{eq:eventual-tail-potential} holds on every
tail beyond $\xi_h(R_0)$;
\item the norm-one tail frames used in the identity are valid beyond
$\xi_h(R_0)$; and
\item no fixed metric breakpoint lies beyond $\xi_h(R_0)$.
\end{enumerate}
The semistable refinement and this finite metric-graph subdivision are fixed
as $R$ varies.  The threshold $R_0$ is attached to the finite collection of
objects and choices fixed in the statement under consideration (divisors or
sections, metrics, tail frames, map, and chosen subdivision).  For a varying target value
$a\in\mathbf P^1(K)$ we write $R_0(a)$; for a fixed finite set
$A=\{a_1,\ldots,a_q\}$ one may take the common threshold
$R_0(A):=\max_j R_0(a_j)$.  Unless a statement specifies otherwise, every
$O(1)$ below is bounded independently of $R\to\infty$ and may depend on the
fixed curve, skeleton, tail coordinates, base point, metrics, map, divisor,
section, and compact subgraph under consideration.

\begin{proposition}[Finite-truncation Poincar\'e-Lelong formula]
\label{thm:truncated-potential-full}
Let $H\subset\cX(K)$ be the fixed nonempty finite puncture set, put
$U:=\cX\setminus H$, and let $\tau:U\to\Sigma$ be the fixed generalized-
skeleton retraction.  For $R>0$, let $\Sigma_R$ be the finite truncation,
let $\xi_h(R)\in\partial\Sigma_R$ be the truncation point on the $h$-tail,
and let $\tau_R:U\to\Sigma_R$ be the finite-truncation retraction.  Let $L$
carry a tail-normalized admissible metric, let $s$ be a nonzero meromorphic
section of $L$, put $U_s:=(-\log\|s\|)|_\Sigma$, and let
$\omega_{L,\Sigma}$ be the skeletal Chern measure.  Choose $R_0>0$ with the property that
for every $R\ge R_0$ the supports of
$\tau_*(\operatorname{div}(s)|_U)$ and $\omega_{L,\Sigma}$ lie in the
interior of $\Sigma_R$ and the terminal affine formula of
Lemma~\ref{lem:eventual-tail-potential} holds beyond each $\xi_h(R_0)$.
Then, for every $R\ge R_0$,
\begin{equation}\label{eq:skeleton-PL-detailed}
\begin{aligned}
 \Delta_{\Sigma_R}(U_s|_{\Sigma_R})
 ={}\omega_{L,\Sigma}|_{\Sigma_R}
 -\tau_{R*}\bigl(\operatorname{div}(s)|_U\bigr)
 -\sum_{h\in H}\ord_h(s)\delta_{\xi_h(R)}.
\end{aligned}
\end{equation}
\end{proposition}

\begin{proof}
The choice of $R_0$ places the support of
$\tau_*(\operatorname{div}(s)|_U)$ and that of $\omega_{L,\Sigma}$ in the interior of
$\Sigma_R$.  The skeletal Poincar\'e--Lelong formula
\eqref{eq:skeletal-poincare-lelong} therefore gives the asserted
coefficients at every point of
$\Sigma_R\setminus\partial\Sigma_R$.

Fix $h\in H$.  Let $v_h^{\mathrm{out}}\in
T_{\xi_h(R)}\Sigma$ be the tangent direction that points from
$\xi_h(R)$ toward the deleted part of $\rho_h$.  Restriction from $\Sigma$
to $\Sigma_R$ removes exactly this direction, so
\[
\begin{aligned}
 (\Delta_{\Sigma_R}U_s)(\xi_h(R))
 &=\sum_{v\in T_{\xi_h(R)}\Sigma_R}\partial_vU_s(\xi_h(R))\\
 &=\sum_{v\in T_{\xi_h(R)}\Sigma}\partial_vU_s(\xi_h(R))
   -\partial_{v_h^{\mathrm{out}}}U_s(\xi_h(R))\\
 &=(\Delta_\Sigma U_s)(\xi_h(R))-\ord_h(s).
\end{aligned}
\]
The final equality uses \eqref{eq:eventual-tail-potential} of Lemma \ref{lem:eventual-tail-potential}, whose derivative
in the outward coordinate is $\ord_h(s)$.  Since the choice of $R_0$
excludes Chern and divisor mass from $\xi_h(R)$, the first term in the last
line is zero.  Thus restriction contributes
$-\ord_h(s)\delta_{\xi_h(R)}$.  Summing this calculation over $h\in H$
gives \eqref{eq:skeleton-PL-detailed}.
\end{proof}

\begin{definition}[Multi-tail Dirichlet Green function ]\label{def:green}
For every $R>0$, the multi-tail Dirichlet Green function is the continuous
piecewise-affine function $G_R$ on $\Sigma_R$ characterized by the following
Dirichlet problem:
\begin{enumerate}[label=(\roman*)]
\item $G_R=0$ on $\partial\Sigma_R$;
\item $G_R$ is harmonic on $\Sigma_R\setminus(\{b\}\cup\partial\Sigma_R)$;
\item the coefficient of $\Delta_{\Sigma_R}G_R$ at the base point $b$ is $-n_H$.
\end{enumerate} Its boundary normal derivative measure is
\begin{equation}\label{eq:green-lap}
        \nu_R:=\Delta_{\Sigma_R}G_R+n_H\delta_b.
\end{equation}
\end{definition}

Lemma~\ref{lem:green-asymptotics} shows that $\nu_R$ is a nonnegative measure supported on $\partial\Sigma_R$, of total mass $n_H.$

\begin{lemma}[Existence, positivity, and asymptotics]\label{lem:green-asymptotics}
Let $H\subset\cX(K)$ be the fixed nonempty finite puncture set
$(n_H:=\#H=\deg H\ge1)$, let
\[
        \Sigma=\Gamma_0\cup\bigcup_{h\in H}\rho_h
\]
be the fixed generalized skeleton with base point $b\in\Gamma_0$, and let
$\Sigma_R$ and $\xi_h(R)\in\partial\Sigma_R$ be the finite truncation and
its boundary points.  Then for every $R>0$, the multi-tail Dirichlet problem of
Definition~\ref{def:green} has a unique solution $G_R$, and $G_R\ge0$.
Moreover, $\nu_R=\Delta_{\Sigma_R}G_R+n_H\delta_b$
has the unique representation
\[
        \nu_R
        =\sum_{h\in H}\lambda_h(R)\delta_{\xi_h(R)},
        \qquad
        \lambda_h(R)\ge0,
        \qquad
        \sum_{h\in H}\lambda_h(R)=n_H.
\]
In addition, for every $h\in H$,
\[
        \lambda_h(R)=1+O(R^{-1})
        \qquad(R\to\infty),
\]
where the implied constant may depend on $h$ and the fixed source-side choices, but
not on $R$. Finally, for every fixed compact subgraph $C\subset\Sigma$,
\[
        G_R(x)=R+O_C(1)
\]
uniformly for $x\in C$; equivalently, there are constants $A_C,R_C>0$ such
that
\[
        \sup_{x\in C}|G_R(x)-R|\le A_C
        \qquad(R\ge R_C).
\]
\end{lemma}

\begin{proof}
Fix $R>0$ and set
\[
 B:=\partial\Sigma_R=\{\xi_h(R):h\in H\}.
\]
Choose a finite metric-graph subdivision whose vertex set $V$ contains $b$,
every point of $B$, every branch point, and one auxiliary interior point on
each loop edge. Thus no edge of the resulting finite graph has identical
endpoints. No point added for this subdivision is required to be a
semistable vertex. Write
\[
        V=I\sqcup B,
        \qquad I:=V\setminus B.
\]
For an edge $e$ of length $\ell_e>0$, let $Q=(Q_{vw})_{v,w\in V}$ be the
standard weighted graph Laplacian \cite[Section~5]{BakerFaberGraphs},
\[
 Q_{vv}=\sum_{e\ni v}\frac1{\ell_e},
 \qquad
 Q_{vw}=-\sum_{e\text{ joins }v,w}\frac1{\ell_e}
 \quad(v\ne w).
\]
For every $z=(z_v)_{v\in V}\in\R^V$,
\begin{equation}\label{eq:dirichlet-matrix-energy}
 z^{\mathsf T}Qz
 =
 \sum_{e=vw}\frac{(z_v-z_w)^2}{\ell_e}\geq0.
\end{equation}
Indeed,
\[
 z^{\mathsf T}Qz
 =\sum_{v\in V}Q_{vv}z_v^2+
  \sum_{\substack{v,w\in V\\v\ne w}}Q_{vw}z_vz_w.
\]
An edge $e=vw$ contributes
\[
 \frac{z_v^2}{\ell_e}+\frac{z_w^2}{\ell_e}
 -\frac{2z_vz_w}{\ell_e}
 =\frac{(z_v-z_w)^2}{\ell_e}.
\]
Summing these contributions proves \eqref{eq:dirichlet-matrix-energy}.

If vertex values $F=(F(v))_{v\in V}$ are extended affinely along every
edge, then
\begin{equation}\label{eq:matrix-laplacian-sign}
\begin{aligned}
 (QF)_v
 &=Q_{vv}F(v)+\sum_{w\ne v}Q_{vw}F(w)\\
 &=\sum_{e=vw}\frac{F(v)-F(w)}{\ell_e}.
\end{aligned}
\end{equation}
The outgoing slope from $v$ toward $w$ along $e$ is
$(F(w)-F(v))/\ell_e$.  Hence
\[
 (\Delta_{\Sigma_R}F)(v)
 =\sum_{e=vw}\frac{F(w)-F(v)}{\ell_e}
 =-(QF)_v.
\]
Impose $G_R|_B=0$, and let
\[
        u=(G_R(v))_{v\in I}\in\R^I
\]
be the vector of unknown interior values. With respect to
$V=I\sqcup B$, write
\[
 Q=
 \begin{pmatrix}
 Q_{II}&Q_{IB}\\
 Q_{BI}&Q_{BB}
 \end{pmatrix}.
\]
At $v\in I\setminus\{b\}$ the required Laplacian coefficient is zero, and
at $b$ it is $-n_H$.  Since the boundary values are zero and
$(\Delta_{\Sigma_R}G_R)(v)=-(QG_R)_v$, these requirements are equivalent to
\begin{equation}\label{eq:dirichlet-principal-system}
        Q_{II}u=n_He_b,
\end{equation}
where $e_b\in\R^I$ is the standard basis vector indexed by $b$.
The sign is positive because
$-(QG_R)_b=-n_H$ is equivalent to $(QG_R)_b=n_H$.

The matrix $Q_{II}$ is positive definite. Indeed, for $0\ne u\in\R^I$,
extend $u$ by zero on $B$ and write $z=(u,0)\in\R^V$. Equation
\eqref{eq:dirichlet-matrix-energy} gives
\[
 u^{\mathsf T}Q_{II}u
 =
 z^{\mathsf T}Qz
 =
 \sum_{e=vw}\frac{(z_v-z_w)^2}{\ell_e}.
\]
If this sum were zero, then $z_v=z_w$ on every edge. Connectedness of
$\Sigma_R$ would make $z$ constant on $V$. Since $B\ne\varnothing$ and
$z|_B=0$, this constant would be zero, contrary to $u\ne0$. Hence
\[
        u^{\mathsf T}Q_{II}u>0
        \qquad(0\ne u\in\R^I).
\]
Therefore $\ker Q_{II}=\{0\}$. Since $Q_{II}$ is a finite square matrix, it
is invertible, and
\[
        u=n_HQ_{II}^{-1}e_b
\]
is the unique solution of \eqref{eq:dirichlet-principal-system}. Extending
these vertex values affinely on every edge produces a unique continuous
piecewise-affine function $G_R$ satisfying the stated Dirichlet problem.

We next prove positivity. Suppose that $G_R$ had a negative minimum at an
interior point $x_0$. Along every outgoing germ the function cannot decrease,
so
\[
 \partial_vG_R(x_0)\geq0
 \qquad(v\in T_{x_0}\Sigma_R).
\]
At $x_0=b$ this would give
\[
 \Delta_{\Sigma_R}G_R(b)\geq0,
\]
contrary to $\Delta_{\Sigma_R}G_R(b)=-n_H$. If $x_0\ne b$, harmonicity and
the preceding inequalities give
\[
 0=\Delta_{\Sigma_R}G_R(x_0)
   =\sum_{v\in T_{x_0}\Sigma_R}\partial_vG_R(x_0),
 \qquad
 \partial_vG_R(x_0)=0\quad\text{for every }v.
\]
On every adjacent edge, affine linearity and the zero initial slope show that
$G_R$ remains equal to the same minimum up to the next vertex.  That vertex
is again a global minimum.  Unless it is $b$ or belongs to $B$, the same
harmonicity argument forces zero slope on every incident edge.  Because the
subdivided graph is finite and connected and $B\ne\varnothing$, a path from
$x_0$ reaches either $b$ or $B$.  The first possibility contradicts the
case $x_0=b$, and the second contradicts $G_R|_B=0$. Hence $G_R\geq0$.

By Definition \ref{def:green}, 
\[
 \nu_R(\{b\})=-n_H+n_H=0.
\]
At every other interior point $x\ne b$, harmonicity gives
$(\Delta_{\Sigma_R}G_R)(x)=0$ and $\delta_b(\{x\})=0$.  Therefore
\[
        \supp\nu_R\subseteq B.
\]
Since $\supp\nu_R\subseteq B$, there are unique coefficients
$\lambda_h(R)$ such that
\[
        \nu_R=\sum_{h\in H}\lambda_h(R)\delta_{\xi_h(R)}.
\]
At a boundary vertex $\xi_h(R)$, the unique direction in $\Sigma_R$ points
toward the interior. Since $G_R(\xi_h(R))=0$ and $G_R\geq0$, its outgoing
slope in this direction is nonnegative. Hence
\[
        \lambda_h(R)\geq0
        \qquad(h\in H).
\]
The total-mass formula for a graph Laplacian gives
\[
 0=(\Delta_{\Sigma_R}G_R)(\Sigma_R)
   =\sum_{h\in H}\lambda_h(R)-n_H,
\]
and therefore
\begin{equation}\label{eq:green-boundary-total-mass}
        \sum_{h\in H}\lambda_h(R)=n_H.
\end{equation}

Let $C_0=\Gamma_0$ be the compact core, after subdivision at $b$ and at all
tail attachment points. Thus
\[
        \{v_h\}=C_0\cap\rho_h,
        \qquad
        u_R:=G_R|_{C_0}.
\]
The normalization
\[
        x_h(v_h)=0,\qquad x_h(\xi_h(R))=R
\]
shows that the segment $[v_h,\xi_h(R)]$ has length exactly $R$. Since $G_R$
is affine on its interior and vanishes at $\xi_h(R)$,
\begin{equation}\label{eq:green-exact-tail}
 G_R(\xi_h(t))
 =
 u_R(v_h)\left(1-\frac{t}{R}\right),
 \qquad 0\leq t\leq R.
\end{equation}
Its boundary mass and its outgoing slope at $v_h$ are therefore
\begin{equation}\label{eq:green-exact-boundary-weight}
 \lambda_h(R)=\frac{u_R(v_h)}{R},
 \qquad
 \partial_{\rho_h}G_R(v_h)=-\lambda_h(R).
\end{equation}
Harmonicity at each attachment point, together with the prescribed mass at
$b$, gives the measure identity on $C_0$
\begin{equation}\label{eq:green-core-equation}
 \Delta_{C_0}u_R
 =
 \sum_{h\in H}\lambda_h(R)\delta_{v_h}
 -n_H\delta_b.
\end{equation}

For $h\in H$ with \(v_h\ne b\), let $g_{v_h,b}$ be the two-point Green
kernel normalized by
\[
 \Delta_{C_0}g_{v_h,b}=\delta_b-\delta_{v_h},
 \qquad
 g_{v_h,b}(b)=0.
\]
If \(v_h=b\), set \(g_{v_h,b}:=0\); the same displayed identities then
hold because \(\delta_b-\delta_{v_h}=0\).
Equations \eqref{eq:green-boundary-total-mass} and
\eqref{eq:green-core-equation} imply
\[
 \Delta_{C_0}
 \left(
 u_R+\sum_{h\in H}\lambda_h(R)g_{v_h,b}
 \right)=0.
\]
The compact maximum principle and the normalization at $b$ yield
\begin{equation}\label{eq:green-core-representation}
 u_R(x)
 =
 u_R(b)-\sum_{h\in H}\lambda_h(R)g_{v_h,b}(x).
\end{equation}
Since $C_0$ is compact and $H$ is finite, there is $M>0$, independent of
$R$, such that
\[
 0\leq g_{v_h,b}(x)\leq M
 \qquad(x\in C_0,\ h\in H).
\]
Using $\sum_h\lambda_h(R)=n_H$ in
\eqref{eq:green-core-representation}, we obtain
\begin{equation}\label{eq:green-core-oscillation}
        |u_R(x)-u_R(b)|\leq Mn_H
        \qquad(x\in C_0).
\end{equation}
On the other hand, \eqref{eq:green-exact-boundary-weight} and
\eqref{eq:green-boundary-total-mass} give
\[
        n_HR=\sum_{h\in H}u_R(v_h).
\]
Equation \eqref{eq:green-core-oscillation} then gives
\[
 \left|n_HR-n_Hu_R(b)\right|\le n_H^2M,
\]
and hence
\[
        |u_R(b)-R|\le n_HM.
\]
Combining this bound with \eqref{eq:green-core-oscillation} gives, uniformly
for $x\in C_0$,
\[
 |u_R(x)-R|\le 2n_HM,\qquad
 |u_R(v_h)-R|\le 2n_HM.
\]
By \eqref{eq:green-exact-boundary-weight},
\[
 \lambda_h(R)=1+O(R^{-1})
 \qquad(R\to\infty),
\]
where the implied constant is independent of $h$ because $H$ is finite.

Finally, let $C\subset\Sigma$ be a fixed compact subgraph. There is a constant
$M_C$ such that every $x=\xi_h(t)\in C\cap\rho_h$ satisfies
$0\leq t\leq M_C$. For $R>M_C$, \eqref{eq:green-exact-tail} and the
preceding uniform bound give
\[
 \left|G_R(\xi_h(t))-R\right|
 \le 2n_HM+M_C+\frac{2n_HMM_C}{R}.
\]
Together with the estimate on $C_0$, this proves that there is a constant
$A_C$, independent of sufficiently large $R$, such that
\[
        |G_R(x)-R|\le A_C
        \qquad(x\in C).
\]
\end{proof}

\section{The first main theorem for morphisms of Berkovich curves}\label{sec:first-main-dirichlet}\label{sec:auxiliary-choice-invariance}\label{sec:function-theoretic-fmt}

\subsection{Characteristic functions of morphisms}\label{sec:nevanlinna-functions}
Let \(\cX, \cY\) be two proper connected quasi-smooth strictly \(K\)-analytic
Berkovich curves and let
\[
        f:\cX\longrightarrow\cY
\]
be a nonconstant finite analytic morphism.  Let \(D\) be an effective analytic
Cartier divisor on \(\cY\),  and let
\(s_D\) denote the canonical nonzero analytic section of
\(\cO_{\cY}(D)\), so that
\[
        \operatorname{div}(s_D)=D.
\]
Fix a piecewise-linear metric on \(\cO_{\cY}(D)\) whose pullback to
\[
        L_{f,D}:=f^*\cO_{\cY}(D)
\]
is tail-normalized admissible relative to \((\Sigma,H)\), and put
\[
        u_{f,D}:=(-\log\|f^*s_D\|)|_\Sigma.
\]
\begin{definition}[Nevanlinna functions]\label{def:morphism-nevanlinna-functions} For every \(R>0\), define the characteristic, counting, and proximity
functions associated with \((f,D)\), respectively by
\begin{equation}\label{eq:morphism-fmt-defs}
\begin{aligned}
 T_f(D,R)&:=\int_{\Sigma_R}G_R\,d\omega_{L_{f,D},\Sigma},\\
 N_f(D,R)&:=\int_{\Sigma_R}G_R\,d\tau_{R*}(f^*D|_U),\\
 m_f(D,R)&:=\int_{\partial\Sigma_R}u_{f,D}\,d\nu_R.
\end{aligned}
\end{equation}
\end{definition}

\begin{remark}[Semipositive choice for effective target divisors]
\label{rem:semipositive-target-metric}
For the effective divisor $D$ in
Definition~\ref{def:morphism-nevanlinna-functions}, the principal metric may
be chosen semipositive. Choose a semistable formal model of $\cY$ on which
the horizontal closure $\mathfrak D$ of $D$ is Cartier and meets the special
fibre in its smooth locus. Its canonical section restricts to a nonzero
regular section on every irreducible component of the special fibre.
Thus $\mathcal O(\mathfrak D)$ has nonnegative degree on each component
and is relatively nef. It induces a semipositive formal metric; see
\cite[Proposition~3.10]{Vilsmeier}. Its formal pullback is again nef by the
degree projection formula, so the induced metric on
$L_{f,D}=f^*\cO_{\cY}(D)$ is semipositive. Formal frames have norm one near
each puncture; after a common semistable refinement, this metric is
tail-normalized admissible relative to $(\Sigma,H)$. Consequently, by
\cite[Remark~4.10]{Vilsmeier} and
\eqref{eq:chern-total-degree-generalized},
\[
 \omega_{L_{f,D},\Sigma}\geq0,
 \qquad
 \omega_{L_{f,D},\Sigma}(\Sigma)=\deg(f)\deg(D).
\]
Since $G_R\geq0$,
\[
 T_f(D,R)=\int_{\Sigma_R}G_R\,d\omega_{L_{f,D},\Sigma}\geq0.
\]
This choice concerns only the principal target-divisor term. The auxiliary
metrics on canonical and logarithmic-Jacobian line bundles used later
remain unrestricted and may have signed skeletal first Chern measures.
\end{remark}

The preceding definition uses fixed tail coordinates and a fixed
tail-normalized admissible metric. Their changes have the following
bounded effect.

\begin{proposition}\label{prop:nevanlinna-choice-stability}
Fix $f$ and $D$ as in Definition~\ref{def:morphism-nevanlinna-functions}.
Compare two admissible systems on the same generalized-skeleton ends, after
passage to a common ordinary metric-graph refinement, such that on each common
terminal subray the normalized tail coordinates satisfy
\[
        x_h'=x_h+c_h,\qquad c_h\in\mathbb R.
\]
Allow also two choices of piecewise-linear metric on $\mathcal O_{\cY}(D)$
whose pullbacks to the fixed line bundle
$L_{f,D}=f^*\mathcal O_{\cY}(D)$ are tail-normalized admissible, as in
Definition~\ref{def:morphism-nevanlinna-functions}.  Let
\[
 T_f^{(i)}(D,R),\qquad N_f^{(i)}(D,R),\qquad m_f^{(i)}(D,R)
 \qquad(i=1,2)
\]
denote the three functions computed from the two systems.  Then
\[
\begin{aligned}
 T_f^{(2)}(D,R)-T_f^{(1)}(D,R)&=O(1),\\
 N_f^{(2)}(D,R)-N_f^{(1)}(D,R)&=O(1),\\
 m_f^{(2)}(D,R)-m_f^{(1)}(D,R)&=O(1)
\end{aligned}
\qquad(R\to\infty).
\]
If the truncation system is fixed and only the metric is changed, then
$N_f^{(2)}(D,R)=N_f^{(1)}(D,R)$ exactly for every $R>0$.  Ordinary
metric-graph subdivision leaves all three functions unchanged exactly.
\end{proposition}

\begin{lemma}\label{lem:tail-coordinate-translation}
Let $H\subset\cX(K)$ be the fixed nonempty puncture set.  Let two admissible
choices of compact core and normalized tail coordinates be given on the same
generalized-skeleton ends indexed by $h\in H$.  After
passage to a common skeletal refinement, suppose that on each common
terminal subray
\[
        x_h'=x_h+c_h,
        \qquad c_h\in\mathbb R.
\]
Let $\Sigma_R$ and $\Sigma_R'$ be the corresponding finite truncations,
and let $(G_R,\nu_R)$ and $(G_R',\nu_R')$ be the associated multi-tail
Dirichlet Green functions and boundary measures, respectively.  For every fixed finite signed measure $\mu$
supported on a compact subgraph of the common refinement,
\[
        \int G_R'\,d\mu-\int G_R\,d\mu=O(1)
        \qquad(R\to\infty).
\]
If $L$ is a fixed analytically metrized line bundle and $s$ is a fixed
nonzero meromorphic section whose potential is eventually affine of slope
$\ord_h(s)$ on each $h$-tail in both coordinate systems, then the
corresponding boundary pairings also satisfy
\[
 \int_{\partial\Sigma_R'}(-\log\|s\|)\,d\nu_R'
 -\int_{\partial\Sigma_R}(-\log\|s\|)\,d\nu_R
 =O(1).
\]
\end{lemma}

\begin{proof}
Let \(C_*\) be a fixed compact subgraph of the common refinement containing
the support of \(\mu\), the relevant Chern measures, and the retracted
divisors. Lemma~\ref{lem:green-asymptotics}, applied to the two systems, gives
\[
 G_R|_{C_*}=R+O(1),
 \qquad
 G_R'|_{C_*}=R+O(1).
\]
Thus there are constants \(A,A'>0\), independent of sufficiently large
\(R\), such that
\[
 \sup_{x\in C_*}|G_R(x)-R|\le A,
 \qquad
 \sup_{x\in C_*}|G'_R(x)-R|\le A'.
\]
Consequently
\[
 \left|\int_{C_*}(G_R'-G_R)\,d\mu\right|
 \leq \sup_{x\in C_*}|G_R'(x)-G_R(x)|\,|\mu|(C_*)
 \le (A+A')|\mu|(C_*).
\]
On the \(h\)-tail the two boundary points have uniformly bounded distance. Choose
$T_h\ge0$ and $\beta_h\in\mathbb R$ such that
\[
 U_s(\xi_h(t))=\ord_h(s)t+\beta_h
\]
for $t\ge T_h$. Take $R_0>\max_{h\in H}(T_h+|c_h|)$, so that for $R\ge R_0$
both boundary points lie on these affine subrays. Then
\[
 \bigl|U_s(\xi_h'(R))-U_s(\xi_h(R))\bigr|
 \le |\ord_h(s)c_h|.
\]
Moreover, Lemma~\ref{lem:green-asymptotics} gives
\[
 \lambda_h(R)=1+O(R^{-1}),
 \qquad
 \lambda_h'(R)=1+O(R^{-1}).
\]
There is a constant \(B_h\), independent of sufficiently large \(R\), such
that \(|U_s(\xi_h(R))|\le B_h(1+R)\).  Hence
\[
\begin{aligned}
 &\left|
 \lambda_h'(R)U_s(\xi_h'(R))
 -\lambda_h(R)U_s(\xi_h(R))
 \right|\\
 &\qquad\leq
 |\lambda_h'(R)|\,|U_s(\xi_h'(R))-U_s(\xi_h(R))|
 +|\lambda_h'(R)-\lambda_h(R)|\,|U_s(\xi_h(R))|
 \le C_h,
\end{aligned}
\]
for a constant \(C_h\) independent of \(R\).  The set \(H\) is finite, so
summation over \(h\in H\) gives a bound independent of \(R\) for the
difference of the two boundary pairings.
\end{proof}

\begin{lemma}\label{lem:metric-change}
Let $H\subset\cX(K)$ be the fixed nonempty puncture set
$(n_H:=\deg H=\#H\ge1)$ and let
$G_R$ and $\nu_R$ be the multi-tail Dirichlet Green function and boundary
measure of Definition~\ref{def:green} and Lemma~\ref{lem:green-asymptotics}
on the finite truncation $\Sigma_R$.  Let $L$ be an analytic line bundle on
$\cX$, and let $\|\cdot\|_1$ and $\|\cdot\|_2$ be two tail-normalized
admissible metrics on $L$ represented on the same fixed generalized skeleton
$\Sigma$, with skeletal Chern measures $\omega_{L,1}$ and
$\omega_{L,2}$. Then
\[
        \int_{\Sigma_R}G_R\,d\omega_{L,1}
        -\int_{\Sigma_R}G_R\,d\omega_{L,2}=O(1)
        \qquad(R\to\infty).
\]
If $\sigma$ is a fixed nonzero meromorphic section of $L$ and
$U_{\sigma,i}:=(-\log\|\sigma\|_i)|_\Sigma$, then also
\[
 \int_{\partial\Sigma_R}U_{\sigma,1}\,d\nu_R
 -\int_{\partial\Sigma_R}U_{\sigma,2}\,d\nu_R
 =O(1).
\]
The two metrics may use different tail frames and normalization constants.
\end{lemma}

\begin{proof}
In a local frame $e$, set
\[
        \varphi=-\log\|e\|_1+\log\|e\|_2.
\]
Under a change of frame $e'=ae$, both negative logarithmic metric weights change by
$-\log|a|$, so $\varphi$ is independent of $e$.  Equivalently, it is the
logarithm of the ratio of the two metrics on the same line bundle and hence
defines a continuous real-valued function on all of the proper curve
$\cX$.  Since $\cX$ is compact, $\varphi$ is bounded.  After a finite metric-graph subdivision of the fixed $\Sigma$ at the
breakpoints of the local weights of both piecewise-linear metrics,
$\varphi|_\Sigma$ is piecewise affine with only finitely many breakpoints.
On a sufficiently far terminal subray of the $h$-tail it therefore has the
form $a_ht+b_h$.  Boundedness as $t\to+\infty$ forces $a_h=0$.  Thus
$\varphi$ is eventually constant on every ray.

Choose a nonzero meromorphic section $s_0$ of $L$. Applying
Theorem~\ref{thm:skeletal-pl} to $s_0$ with the two metrics and subtracting
cancels the divisor terms. Since
$(-\log\|s_0\|_1)-(-\log\|s_0\|_2)=\varphi|_\Sigma$, this gives
\[
        \omega_{L,1}-\omega_{L,2}=\Delta_\Sigma\varphi.
\]
For $R$ beyond the stabilization point of every tail, $\varphi$ is constant near $\partial\Sigma_R$.  Hence restriction introduces no endpoint-slope measure and
\[
 \Delta_{\Sigma_R}(\varphi|_{\Sigma_R})
 =(\omega_{L,1}-\omega_{L,2})|_{\Sigma_R}.
\]
The graph \(\Sigma_R\) is compact, and \(G_R\) and
\(\varphi|_{\Sigma_R}\) are continuous and piecewise affine on a common
finite subdivision.  By Definition~\ref{def:green} and Lemma~\ref{lem:green-asymptotics},
$G_R$ exists and satisfies
\[
        \Delta_{\Sigma_R}G_R=\nu_R-n_H\delta_b.
\]
Lemma~\ref{lem:graph-ibp-full} therefore gives
\[
\begin{aligned}
        \int_{\Sigma_R}G_R\,d(\omega_{L,1}-\omega_{L,2})
        &=\int_{\Sigma_R}G_R\,d\bigl(\Delta_{\Sigma_R}(\varphi|_{\Sigma_R})\bigr)  \\
        &=\int_{\Sigma_R}\varphi\,d(\Delta_{\Sigma_R} G_R)  \\
        &=\int_{\partial\Sigma_R}\varphi\,d\nu_R-n_H\varphi(b).
\end{aligned}
\]
Since $\nu_R(\partial\Sigma_R)=n_H$,
set
\[
        \|\varphi\|_\infty:=\sup_{x\in\cX}|\varphi(x)|<\infty.
\]
Then
\[
\begin{aligned}
 \left|
 \int_{\Sigma_R}G_R\,d(\omega_{L,1}-\omega_{L,2})
 \right|
 &\leq
 \sup_{x\in\Sigma}|\varphi(x)|\nu_R(\partial\Sigma_R)
 +n_H\sup_{x\in\Sigma}|\varphi(x)|\\
 &\leq2n_H\|\varphi\|_\infty.
\end{aligned}
\]
For the boundary terms,
put $m_i(R):=\int_{\partial\Sigma_R}U_{\sigma,i}\,d\nu_R$ for $i=1,2$. Then
\[
 |m_1(R)-m_2(R)|
 =\left|\int_{\partial\Sigma_R}\varphi\,d\nu_R\right|
 \leq n_H\|\varphi\|_\infty.
\]
Both bounds are independent of $R$.
\end{proof}

\begin{proof}[Proof of Proposition~\ref{prop:nevanlinna-choice-stability}]
On the common refinement, denote the two truncations, Green kernels, and
boundary measures by $\Sigma_R^{(i)}$, $G_R^{(i)}$, and $\nu_R^{(i)}$,
respectively, for $i=1,2$. Set
\[
 L:=L_{f,D},\qquad \sigma:=f^*s_D,\qquad
 u_i:=(-\log\|\sigma\|_i)|_\Sigma,\qquad
 \omega_i:=\omega_{L,i},\qquad
 \mu:=\tau_*(f^*D|_U).
\]
Choose a fixed compact subgraph $C\subset\Sigma$ containing
$\supp\omega_1\cup\supp\mu$. For sufficiently large $R$,
\[
 C\subset (\Sigma_R^{(1)})^\circ\cap(\Sigma_R^{(2)})^\circ.
\]
If $q_R^{(i)}:\Sigma\to\Sigma_R^{(i)}$ are the clamping maps, then
$\tau_R^{(i)}=q_R^{(i)}\circ\tau$ and $q_R^{(i)}|_C=\operatorname{id}_C$.
Consequently,
\[
 \tau_{R*}^{(i)}(f^*D|_U)=(q_R^{(i)})_*\mu=\mu,
 \qquad
 N_f^{(i)}(D,R)=\int_C G_R^{(i)}\,d\mu.
\]

First keep the first metric fixed and compare the truncations. Define
\[
 \widehat T(R):=\int_{\Sigma_R^{(2)}}G_R^{(2)}\,d\omega_1,
 \qquad
 \widehat N(R):=\int_{\Sigma_R^{(2)}}G_R^{(2)}\,d\mu,
 \qquad
 \widehat m(R):=\int_{\partial\Sigma_R^{(2)}}u_1\,d\nu_R^{(2)}.
\]
By Lemma~\ref{lem:tail-coordinate-translation}, there is $A_C\ge0$,
independent of sufficiently large $R$, such that
\[
 \sup_C|G_R^{(2)}-G_R^{(1)}|\le A_C.
\]
It follows that
\[
\begin{aligned}
 |\widehat T(R)-T_f^{(1)}(D,R)|
 &\le A_C|\omega_1|(C)=:A_T,\\
 |\widehat N(R)-N_f^{(1)}(D,R)|
 &\le A_C|\mu|(C)=:A_N.
\end{aligned}
\]
Writing $\xi_h^{(i)}(R)$ and $\lambda_h^{(i)}(R)$ for the corresponding
boundary points and masses, the boundary estimate in the same lemma gives
constants $C_h\ge0$, independent of $R$, with
\[
 \left|
 \lambda_h^{(2)}(R)u_1(\xi_h^{(2)}(R))
 -\lambda_h^{(1)}(R)u_1(\xi_h^{(1)}(R))
 \right|\le C_h.
\]
Thus
\[
 |\widehat m(R)-m_f^{(1)}(D,R)|\le\sum_{h\in H}C_h=:A_m.
\]

Next keep the second truncation fixed and compare the metrics. In a common
local frame $e$, put
\[
 \varphi:=-\log\|e\|_2+\log\|e\|_1.
\]
As in Lemma~\ref{lem:metric-change}, $\varphi$ is frame independent and
continuous on the proper curve $\cX$. Hence
\[
 M:=\sup_{\cX}|\varphi|<\infty,\qquad
 u_2-u_1=\varphi|_\Sigma,\qquad
 \omega_2-\omega_1=\Delta_\Sigma\varphi.
\]
For sufficiently large $R$, $\varphi$ is constant near every boundary point,
so
\[
 \Delta_{\Sigma_R^{(2)}}(\varphi|_{\Sigma_R^{(2)}})
 =(\omega_2-\omega_1)|_{\Sigma_R^{(2)}}.
\]
Integration by parts and
$\Delta_{\Sigma_R^{(2)}}G_R^{(2)}=\nu_R^{(2)}-n_H\delta_b$ give
\[
\begin{aligned}
 T_f^{(2)}(D,R)-\widehat T(R)
 &=\int_{\Sigma_R^{(2)}}G_R^{(2)}\,d(\omega_2-\omega_1)\\
 &=\int_{\partial\Sigma_R^{(2)}}\varphi\,d\nu_R^{(2)}-n_H\varphi(b),\\
 m_f^{(2)}(D,R)-\widehat m(R)
 &=\int_{\partial\Sigma_R^{(2)}}\varphi\,d\nu_R^{(2)}.
\end{aligned}
\]
The counting function contains no metric, so
$N_f^{(2)}(D,R)=\widehat N(R)$. Since
$\nu_R^{(2)}(\partial\Sigma_R^{(2)})=n_H$,
\[
\begin{aligned}
 |T_f^{(2)}(D,R)-T_f^{(1)}(D,R)|&\le 2n_HM+A_T,\\
 |N_f^{(2)}(D,R)-N_f^{(1)}(D,R)|&\le A_N,\\
 |m_f^{(2)}(D,R)-m_f^{(1)}(D,R)|&\le n_HM+A_m.
\end{aligned}
\]
All these bounds are independent of sufficiently large $R$.

If only the metric changes, the same truncation, Green kernel, and
retraction occur in both counting functions; hence they agree for every
$R>0$. Ordinary metric-graph subdivision leaves the underlying metric
space and retractions unchanged. Each new two-valent vertex on an affine
edge contributes zero Laplacian mass. Thus $G_R$, $\nu_R$, the divisor
pushforward, and the Chern measure are unchanged, and so are all three
Nevanlinna functions.
\end{proof}

\subsection{The first main theorem}
Now, we show the first main theorem for finite analytic morphism between Berkovich curves.

\begin{theorem}[First main theorem]\label{thm:morphism-fmt}
Let $f:\cX\to\cY$ be a nonconstant finite analytic morphism between proper
connected quasi-smooth strictly $K$-analytic curves. Let
$H\subset\cX(K)$ be a nonempty finite set, put $U:=\cX\setminus H$ and
$n_H:=\#H$,  and let
$\tau:U\to\Sigma$ and $b\in\Sigma$ be the fixed generalized-skeleton
retraction and base point. Let $D$ be an effective analytic Cartier divisor on $\cY$.
Fix a metric on $\cO_{\cY}(D)$ whose pullback is tail-normalized admissible
relative to the fixed generalized skeleton. Then there exists
$R_0(D)>0$ such that for every $R\ge R_0(D)$,
\begin{equation}\label{eq:morphism-fmt}
        T_f(D,R)
        =N_f(D,R)+m_f(D,R)-n_Hu_{f,D}(b),
\end{equation}where $u_{f, D}:=(-\log\|f^*s_D\|)|_\Sigma.$
\end{theorem}

\begin{proof}
Since
\[
        \operatorname{div}(f^*s_D)=f^*D,
\]
we apply the finite-truncation skeletal Poincar\'e-Lelong formula directly to
the PL-metrized line bundle \(L_{f,D}=f^*\cO_{\cY}(D)\) and its pulled-back
section \(f^*s_D\).  Choose \(R_0(D)>0\) so that, for every
\(R\ge R_0(D)\), the retracted support of \(f^*D|_U\) and the support of
\(\omega_{L_{f,D},\Sigma}\) lie in \(\operatorname{int}(\Sigma_R)\), while
the terminal affine formula for \(u_{f,D}\) and the norm-one tail frames are
valid beyond every \(\xi_h(R_0(D))\).

Fix \(R\ge R_0(D)\). Proposition~\ref{thm:truncated-potential-full} gives
\[
\begin{aligned}
 T_f(D,R)-N_f(D,R)
 &=\int_{\Sigma_R}G_R\,
   d\!\left(
      \omega_{L_{f,D},\Sigma}
      -\tau_{R*}(f^*D|_U)
     \right)\\
 &=\int_{\Sigma_R}G_R\,
   d\!\left(
     \Delta_{\Sigma_R}u_{f,D}
      +\sum_{h\in H}\ord_h(f^*s_D)\delta_{\xi_h(R)}
     \right).
\end{aligned}
\]
For every \(h\in H\), the Dirichlet boundary condition gives
\[
 \int_{\Sigma_R}G_R\,d\delta_{\xi_h(R)}
 =G_R(\xi_h(R))=0.
\]
Consequently,
\[
 T_f(D,R)-N_f(D,R)
 =\int_{\Sigma_R}G_R\,d(\Delta_{\Sigma_R}u_{f,D}).
\]
The finite graph \(\Sigma_R\) is compact, and both \(G_R\) and
\(u_{f,D}|_{\Sigma_R}\) are continuous and piecewise affine on a common
finite metric-graph subdivision.  Lemma~\ref{lem:graph-ibp-full} therefore
gives
\[
 \int_{\Sigma_R}G_R\,d(\Delta_{\Sigma_R}u_{f,D})
 =\int_{\Sigma_R}u_{f,D}\,d(\Delta_{\Sigma_R}G_R).
\]
By Definition~\ref{def:green} and Lemma~\ref{lem:green-asymptotics},
\[
        \Delta_{\Sigma_R}G_R=\nu_R-n_H\delta_b,
\]
where \(\nu_R\) is supported on \(\partial\Sigma_R\) and has total mass
\(n_H\).  Hence
\[
\begin{aligned}
 \int_{\Sigma_R}u_{f,D}\,d(\Delta_{\Sigma_R}G_R)
 &=\int_{\Sigma_R}u_{f,D}\,d(\nu_R-n_H\delta_b)\\
 &=\int_{\partial\Sigma_R}u_{f,D}\,d\nu_R-n_Hu_{f,D}(b)\\
 &=m_f(D,R)-n_Hu_{f,D}(b).
\end{aligned}
\]
Combining the preceding equalities and rearranging yields
\eqref{eq:morphism-fmt}.
\end{proof}

Let
\[
        f:\cX\longrightarrow\PP_K^{1,\an}
\]
be a nonconstant finite analytic morphism. Fix a piecewise-linear metric on the
analytic line bundle $\cO_{\PP_K^{1,\an}}(1)$. Its pullback to
\[
        L_f:=f^*\cO_{\PP_K^{1,\an}}(1)
\]
is piecewise linear by \cite[Proposition~3.7(ii)]{Vilsmeier}.  By the
local characterization in \cite[Definition~3.6]{Vilsmeier}, around each
$f(h)$ there is an admissible neighbourhood on which the target line bundle
has a local frame of norm one for this metric.  Pulling such a frame back to
a sufficiently small neighbourhood of $h$ gives a nowhere-vanishing
 norm-one frame for the pullback metric.  Hence, after enlarging the fixed
 semistable refinement to accommodate the finitely many metric breakpoints
 now under consideration, the
pullback metric is tail-normalized admissible relative to $(\Sigma,H)$.

For $a\in\PP^1(K)$, let $s_a$ be a nonzero global analytic section of
$\cO_{\PP_K^{1,\an}}(1)$ whose divisor is $[a]$. Then $f^*s_a$ is an analytic
section of $L_f$ and
\[
        \operatorname{div}(f^*s_a)=f^*[a].
\]
Define the skeletal potential
\[
        u_{f,a}:=(-\log\|f^*s_a\|)|_\Sigma.
\]
Choose an isomorphism
\[
 \varphi_a:\cO_{\PP_K^{1,\an}}([a])
 \xrightarrow{\sim}\cO_{\PP_K^{1,\an}}(1)
\]
carrying the canonical section of $[a]$ to $s_a$, and equip its source with
the metric pulled back by $\varphi_a$. Then $f^*\varphi_a$ is an isometry
from $f^*\cO_{\PP_K^{1,\an}}([a])$ to $L_f$, so the two metrized bundles
have the same skeletal Chern measure. Definition~\ref{def:morphism-nevanlinna-functions}
therefore gives
\[
        T_f([a],R)
        =\int_{\Sigma_R}G_R\,d\omega_{L_f,\Sigma}.
\]
This expression is independent of $a$. The divisor and potential depend on
$a$; we write
\[
        N_f(a,R):=N_f([a],R)
        =\int_{\Sigma_R}G_R\,d\tau_{R*}(f^*[a]|_U),
\]
\[
        m_f(a,R):=m_f([a],R)
        =\int_{\partial\Sigma_R}u_{f,a}\,d\nu_R.
\]

For the principal characteristic we henceforth use the semipositive
tail-normalized metric described in
Remark~\ref{rem:semipositive-target-metric}. Hence $T_f(R)\geq0$ for every
$R$ for which the finite-truncation Green kernel is defined. For arbitrary
auxiliary line bundles, however, the characteristic-type quantity
$T_L(R)$ introduced below may still be signed. The divisor term
$N_f(a,R)$ is nonnegative because $f^*[a]$ is effective.

\begin{corollary}[First main theorem for the projective line]\label{cor:function-fmt-full}
Let $f:\cX\to\mathbf P_K^{1,\an}$ be a nonconstant finite analytic morphism, let
$H\subset\cX(K)$ be the fixed nonempty puncture set
$(n_H:=\deg H=\#H\ge1)$, put $U:=\cX\setminus H$, and let
$\tau:U\to\Sigma$ and $b\in\Sigma$ be the fixed generalized-skeleton
retraction and base point. Then
for every fixed $a\in\mathbf P^1(K)$ there exists a threshold $R_0(a)>0$
such that, for every $R\ge R_0(a)$,
\[
        T_f([a],R)=N_f(a,R)+m_f(a,R)-n_Hu_{f,a}(b).
\]
\end{corollary}

\begin{proof}
Apply Theorem~\ref{thm:morphism-fmt} with
\[
        \cY=\mathbf P_K^{1,\an},\qquad D=[a],
        \qquad \cO_{\cY}(D)\simeq\cO_{\mathbf P_K^{1,\an}}(1),
\]
and with canonical section represented by the chosen section \(s_a\).
Then \(L_{f,D}=L_f\), \(u_{f,D}=u_{f,a}\), and
\[
        \operatorname{div}(f^*s_a)=f^*[a],
\]
so \eqref{eq:morphism-fmt} is exactly the asserted identity.
\end{proof}

\begin{remark}If $s_a'$ is another section with divisor $[a]$, then $s_a'=c s_a$ for some
$c\in K^\times$. Thus $u_{f,a}'=u_{f,a}-\log|c|$. Since $\nu_R$ has total
mass $n_H$, the quantity
\[
        m_f(a,R)-n_Hu_{f,a}(b)
\]
is unchanged.\end{remark}

\section{The second main theorem}\label{sec:second-main}

\subsection{Properties of characteristic functions}\label{sec:target-characteristics}

Retain the notation of Section~\ref{sec:first-main-dirichlet}, and write
$K_{\cX}:=\Omega^1_{\cX/K}$ for the canonical line bundle of $\cX$.

Let $L$ be an analytic line bundle on $\cX$ equipped with a tail-normalized
admissible metric.  We use the  notation
\begin{equation}\label{eq:T-L}
        T_L(R):=\int_{\Sigma_R}G_R\,d\omega_{L,\Sigma}.
\end{equation}
If $Z$ is an effective analytic Cartier divisor on $\cX$, its counting
function inside $U$ is
\begin{equation}\label{eq:N-Z}
        N_Z(R):=\int_{\Sigma_R}G_R\,
        d\tau_{R*}(Z|_U).
\end{equation}
If $A$ is a Cartier divisor on $\cX$, $T_A(R)$ denotes the characteristic of
$\mathcal O_{\cX}(A)$ with a fixed tail-normalized admissible metric.
If $\mathcal L$ is an analytic line bundle on $\cY$ and a fixed
tail-normalized admissible metric has been chosen on $f^*\mathcal L$, set
\begin{equation}\label{eq:target-linebundle-characteristic}
        T_f(\mathcal L,R):=T_{f^*\mathcal L}(R).
\end{equation}
For a Cartier divisor $E$ on $\cY$, we abbreviate
\[
        T_f(E,R):=T_f(\mathcal O_{\cY}(E),R),
\]
with the chosen metric on the pullback understood.  This extends the
characteristic notation from Definition~\ref{def:morphism-nevanlinna-functions}
to target line bundles and Cartier divisors.  When additive notation
such as $K_{\cY}+E$ is used, it denotes the tensor product
$K_{\cY}\otimes\mathcal O_{\cY}(E)$.  Different fixed admissible metric
choices on the same pullback line bundle change these target-side
line-bundle characteristics by at most $O(1)$ by
Lemma~\ref{lem:metric-change}.

For the target divisor $D$ of Definition~\ref{def:morphism-nevanlinna-functions},
with the pullback metric fixed there, this notation agrees exactly with the
previously defined characteristic, and one has
\begin{equation}\label{eq:morphism-auxiliary-identification}
        T_f(D,R)=T_{f^*\mathcal O_{\cY}(D)}(R),
        \qquad
        N_f(D,R)=N_{f^*D}(R).
\end{equation}
\begin{lemma}\label{lem:T-degree}
Let $H\subset\cX(K)$ be the fixed nonempty puncture set
$(n_H:=\deg H=\#H\ge1)$, let $\Sigma_R$ be the finite truncations of the
fixed generalized skeleton, and let $G_R$ be the multi-tail Dirichlet Green
function of Definition~\ref{def:green} and Lemma~\ref{lem:green-asymptotics}.
For every analytic line bundle $L$ on $\cX$ with a fixed tail-normalized
admissible metric and skeletal Chern measure $\omega_{L,\Sigma}$, the
auxiliary characteristic $T_L(R)$ from \eqref{eq:T-L} satisfies
\[
        T_L(R)=\deg(L)R+O(1).
\]
In particular, for the reduced puncture divisor $H=\sum_{h\in H}[h]$,
and for any fixed tail-normalized admissible metrics used to define the two
characteristics,
\[
        T_H(R)=n_HR+O(1),
        \qquad
        T_{K_{\cX}+H}(R)
        =(2g(\cX)-2+n_H)R+O(1).
\]
\end{lemma}

\begin{proof}
Choose a fixed compact subgraph
\[
        C\supseteq\supp\omega_{L,\Sigma}.
\]
By Lemma~\ref{lem:green-asymptotics}, there are constants $A_C,R_C>0$ such
that
\[
        \sup_{x\in C}|G_R(x)-R|\le A_C
        \qquad(R\ge R_C).
\]
By \eqref{eq:chern-total-degree-generalized},
$\omega_{L,\Sigma}(C)=\omega_{L,\Sigma}(\Sigma)=\deg L$.  Hence
\[
\begin{aligned}
 \left|T_L(R)-R\deg L\right|
 &=\left|\int_C(G_R-R)\,d\omega_{L,\Sigma}\right|\\
 &\le A_C\,|\omega_{L,\Sigma}|(C),
\end{aligned}
\]
which is $O(1)$.  Since
\[
        H=\sum_{h\in H}[h],
        \qquad \deg H=n_H,
\]
we obtain $T_H(R)=n_HR+O(1)$.  Finally,
Lemma~\ref{lem:analytic-degree-facts} gives
\[
        \deg(K_{\cX}+H)=2g(\cX)-2+n_H,
\]
and the last assertion follows.
\end{proof}

\begin{corollary}\label{cor:morphism-characteristic-growth}
Under the hypotheses and notation of
Definition~\ref{def:morphism-nevanlinna-functions}, one has
\begin{equation}\label{eq:morphism-characteristic-growth-pullback}
        T_f(D,R)
        =\deg\!\bigl(f^*\mathcal O_{\cY}(D)\bigr)R+O(1).
\end{equation}
In particular,
\begin{equation}\label{eq:morphism-characteristic-growth}
        T_f(D,R)=\deg(f)\deg(D)\,R+O(1),
\end{equation}
where $\deg(D)$ is the analytic divisor degree on $\cY$.
\end{corollary}

\begin{proof}
By \eqref{eq:morphism-auxiliary-identification} and
Lemma~\ref{lem:T-degree},
\[
        T_f(D,R)
        =\deg\!\bigl(f^*\mathcal O_{\cY}(D)\bigr)R+O(1).
\]
Lemma~\ref{lem:analytic-degree-facts} gives
\[
        \deg\!\bigl(f^*\mathcal O_{\cY}(D)\bigr)
        =\deg(f)\deg\mathcal O_{\cY}(D)
        =\deg(f)\deg(D),
\]
which proves \eqref{eq:morphism-characteristic-growth-pullback} and
\eqref{eq:morphism-characteristic-growth}.
\end{proof}

Let 
\[
        f:\cX\longrightarrow\cY
\]
be a finite analytic morphism from $\cX$ to $\cY.$ Let $D$ be a reduced
effective analytic Cartier divisor on $\cY$.  For an
effective divisor $A=\sum_xm_x[x]$ on $U$, put
\[
        A_{\red}:=\sum_{m_x>0}[x].
\]
For the reduced counting function we set
\[
        N_f^{[1]}(D,R):=
        \int_{\Sigma_R}G_R\,
        d\tau_{R*}\bigl((f^*D|_U)_{\red}\bigr).
\]
When $\cY=\mathbf P_K^{1,\an}$ and
$D=a_1+\cdots+a_q$ with distinct $a_j\in\mathbf P^1(K)$, fix a PL metric
on $\mathcal O_{\mathbf P_K^{1,\an}}(1)$.  In the target-line-bundle
notation \eqref{eq:target-linebundle-characteristic}, the standard
projective-line shorthand introduced in
Section~\ref{sec:function-theoretic-fmt} is precisely
\[
        T_f\bigl(\mathcal O_{\mathbf P_K^{1,\an}}(1),R\bigr)
        =T_{f^*\mathcal O_{\mathbf P_K^{1,\an}}(1)}(R).
\]
The counting shorthand introduced in
Section~\ref{sec:function-theoretic-fmt} satisfies
\[
        N_f(a_j,R)=N_{f^*[a_j]}(R).
\]
We define its level-one truncation by
\[
        N_f^{[1]}(a_j,R):=
        \int_{\Sigma_R}G_R\,
        d\tau_{R*}\bigl((f^*[a_j]|_U)_{\red}\bigr).
\]
Thus, for reduced $D=a_1+\cdots+a_q$,
\[
        N_f^{[1]}(D,R)=\sum_{j=1}^qN_f^{[1]}(a_j,R).
\]

\begin{lemma}\label{lem:T-additivity}
Let $L_1$ and $L_2$ be analytic line bundles on $\cX$, and form all
characteristics with the same multi-tail Green function $G_R$ on the fixed
truncations $\Sigma_R$.  If $L_1$ and $L_2$ are equipped with compatible
tail-normalized admissible metrics and $L_1\otimes L_2$ with the induced
tensor-product metric, then
\[
        T_{L_1\otimes L_2}(R)=T_{L_1}(R)+T_{L_2}(R).
\]
For arbitrary fixed tail-normalized admissible metrics on these line bundles,
the equality holds up to $O(1)$. In particular, for any fixed
tail-normalized admissible metrics on $L$ and $L^{-1}$,
\[
        T_{L^{-1}}(R)=-T_L(R)+O(1).
\]
\end{lemma}

\begin{proof}
Choose a common semistable formal model representing the two metrics
by formal line bundles $\mathfrak L_1$ and $\mathfrak L_2$.
For compatible tensor-product metrics, $\mathfrak L_1\otimes\mathfrak L_2$
represents the induced metric. On each special-fibre component $Y_v$,
\[
 \deg\bigl((\mathfrak L_1\otimes\mathfrak L_2)|_{Y_v}\bigr)
 =\deg(\mathfrak L_1|_{Y_v})+\deg(\mathfrak L_2|_{Y_v});
\]
see \cite[Chapter~7]{Liu}. Definition~\ref{def:skeletal-chern} therefore gives
\[
        \omega_{L_1\otimes L_2,\Sigma}
        =\omega_{L_1,\Sigma}+\omega_{L_2,\Sigma},
\]
hence the first equality after pairing with $G_R$.  For arbitrary fixed
metrics, Lemma~\ref{lem:metric-change} changes the two sides by $O(1)$.
For the last assertion, apply the compatible-metric identity to
$L\otimes L^{-1}\simeq\mathcal O_{\cX}$ and equip the trivial bundle with
its unit metric, whose local weight and characteristic are identically zero.
Lemma~\ref{lem:metric-change} replaces any other fixed tail-normalized metric
on the trivial bundle at a cost of $O(1)$.  Thus
$T_{L^{-1}}(R)=-T_L(R)+O(1)$.
\end{proof}

\subsection{Jacobian sections and ramification}\label{sec:log-jac}

We consider finite generically \`etale analytic morphisms
$f:\cX\to\cY$. The differential $df$ defines the ramification divisor;
the finite-truncation Green functions then give the second main theorem.
\begin{remark}[Generic \'etaleness]\label{rem:generic-etale}
A finite morphism \(f:\cX\to\cY\) is \emph{generically \'etale} if there
is a nonempty dense analytic open subset \(V\subset\cX\) on which \(f\)
is \'etale, equivalently on which the differential
\(df:f^*K_{\cY}\to K_{\cX}\) is an isomorphism. Since \(f\) is finite and
nonconstant, such a \(V\) exists precisely when \(df\) is not the zero
morphism. This condition is needed only from the present section onward.
It is what makes the Jacobian section
\[
 J_f\in H^0\bigl(\cX,K_{\cX}\otimes f^*K_{\cY}^{-1}\bigr)
\]
a \emph{nonzero} section, hence the ramification divisor
\(R_f=\operatorname{div}(J_f)\) a well-defined finite effective divisor and
the logarithmic Jacobian \(\sigma_{f,D}=f^*s_D/J_f\) a well-defined
meromorphic section of \(L^{\log}_{f,D}\). It is also what makes the local
ramification inequality \(r_x\ge e_x-1\) of
Lemma~\ref{lem:local-ramification} and the degree form of
Riemann-Hurwitz \eqref{eq:riemann-hurwitz} available. Without it, \(J_f\)
would vanish identically, \(R_f\) would be undefined, and the second main
theorem would have no ramification term to estimate.
\end{remark}

Quasi-smoothness implies that $K_{\cX}$ and $K_{\cY}$ are analytic line bundles. The analytic
morphism $f$ induces
\[
        df:f^*K_{\cY}\longrightarrow K_{\cX}.
\]
Because $f$ is generically \`etale, there is a nonempty dense analytic open
subset $V\subset\cX$ on which $f$ is \`etale.  On $V$ the exact sequence of
analytic K\"ahler differentials identifies
\[
        f^*\Omega^1_{\cY/K}|_V
        \xrightarrow{\sim}
        \Omega^1_{\cX/K}|_V.
\]
Hence $df$ is not the zero homomorphism.

Since both sides have rank one,
\[
\mathcal Hom_{\mathcal O_{\cX}}(f^*K_{\cY},K_{\cX})
\simeq
K_{\cX}\otimes f^*K_{\cY}^{-1}.
\]
Thus $df$ determines a nonzero analytic section
\begin{equation}\label{eq:J-section}
        J_f\in H^0\!\left(
        \cX,K_{\cX}\otimes f^*K_{\cY}^{-1}\right).
\end{equation}

\begin{definition}[Jacobian section and Cartier ramification divisor]
The section $J_f$ in \eqref{eq:J-section} is the analytic Jacobian section of
$f$.  Its effective analytic Cartier divisor
\begin{equation}\label{eq:jacobian-divisor}
        R_f:=\operatorname{div}(J_f)
\end{equation}
is its ramification divisor.
\end{definition}

The divisor $R_f$ records the type-I divisor of the analytic differential.
At a type-I point this coefficient agrees with the differential index of
Cohen-Temkin-Trushin; see
\cite[Remark~4.5.6(1) and Theorem~4.6.4]{CTT}.

Since $J_f$ is a nonzero analytic section of
$K_{\cX}\otimes f^*K_{\cY}^{-1}$, Lemma~\ref{lem:analytic-degree-facts}
gives
\[
\begin{aligned}
 \deg R_f
 &=\deg\!\left(K_{\cX}\otimes f^*K_{\cY}^{-1}\right)\\
 &=\deg K_{\cX}-\deg(f^*K_{\cY})\\
 &=2g(\cX)-2-(\deg f)\bigl(2g(\cY)-2\bigr).
\end{aligned}
\]
Equivalently,
\begin{equation}\label{eq:riemann-hurwitz}
        2g(\cX)-2
        =(\deg f)\bigl(2g(\cY)-2\bigr)+\deg R_f.
\end{equation}
This is the degree form of Riemann-Hurwitz in the present analytic
setting; compare the intrinsic formula for finite morphisms of compact
Berkovich curves in \cite[Theorem~6.2.3]{CTT}.

\begin{definition}[Tame and wild ramification at type-I points]
\label{def:type-one-tame-wild}
\cite[Chapter~IV, \S2, p.~299]{Hartshorne}.
Let $f:\cX\to\cY$ be a finite generically \`etale analytic morphism between
proper connected quasi-smooth strictly $K$-analytic Berkovich curves.
For a type-I point $x\in\cX$, set $y:=f(x)$ and let $e_x$ be the local degree at $x$.
The analytic local rings $\mathcal O_{\cY,y}\subset\mathcal O_{\cX,x}$ are
discrete valuation rings with residue field $K$. Ramification at $x$ is
\emph{tame} if $\operatorname{char}K=0$, or if
$\operatorname{char}K=p>0$ and $p\nmid e_x$. It is \emph{wild} if
$\operatorname{char}K=p>0$ and $p\mid e_x$.
\end{definition}

Here tame and wild refer to the type-I local-ring extensions in this
definition.

\begin{lemma}\label{lem:local-ramification}
Let $f:\cX\to\cY$ be a finite generically \`etale analytic morphism between
proper connected quasi-smooth strictly $K$-analytic Berkovich curves, let
$J_f$ be its Jacobian section, and let
$R_f:=\operatorname{div}(J_f)$.  For a type-I point $x\in\cX(K)$, put
$y:=f(x)$, let $e_x$ be the local degree of $f$ at $x$, and put
\[
        r_x:=\ord_x(R_f).
\]
Then
\[
        r_x\ge e_x-1.
\]
If the image of the integer $e_x$ in $K$ is nonzero, then $r_x=e_x-1$.
\end{lemma}

\begin{proof}
The analytic local rings at the type-I points $x$ and $y=f(x)$ are
discrete valuation rings; choose local parameters $t$ at $x$ and $w$ at
$y$, as in \cite[\S3.3.7]{CTT}.  By the definition of the local degree,
\[
        f^*w=t^{e_x}u,
        \qquad u\in\mathcal O_{\cX,x}^{\times}.
\]
Since $\Omega^1_{\cX/K,x}$ is generated by $dt$, write $du=u'\,dt$ locally.
Then
\[
        d(f^*w)
        =t^{e_x-1}(e_xu+tu')\,dt.
\]
The coefficient of the Jacobian divisor is the order of this differential
coefficient; equivalently this is the local differential index of
\cite[Lemma~4.6.2 and Theorem~4.6.4]{CTT}.  Hence
\[
        r_x=e_x-1+\ord_t(e_xu+tu').
\]
Generically \`etale implies that the differential is not identically zero,
so $e_xu+tu'$ is a nonzero element of the local discrete valuation ring and its order is
nonnegative.  Thus $r_x\ge e_x-1$.  If $e_x\ne0$ in $K$, then
\[
        (e_xu+tu')\bmod(t)=e_xu(x)\ne0,
\]
so the final order is zero and $r_x=e_x-1$.
\end{proof}

The reduced effective analytic Cartier divisor $D$ determines
$\mathcal O_{\cY}(D)$ and its canonical analytic section
\[
        s_D\in H^0\bigl(\cY,\mathcal O_{\cY}(D)\bigr),
        \qquad \operatorname{div}(s_D)=D.
\]

\begin{definition}\label{def:log-jac-section}
Define
\[
 \sigma_{f,D}:=\frac{f^*s_D}{J_f}.
\]
Since $f^*s_D$ is an analytic section of $f^*\mathcal O_{\cY}(D)$ and
$J_f$ is an analytic section of $K_{\cX}\otimes f^*K_{\cY}^{-1}$,
$\sigma_{f,D}$ is a meromorphic section of
\[
 f^*\mathcal O_{\cY}(D)\otimes
 \bigl(K_{\cX}\otimes f^*K_{\cY}^{-1}\bigr)^{-1}
 \simeq f^*\bigl(K_{\cY}\otimes\mathcal O_{\cY}(D)\bigr)
 \otimes K_{\cX}^{-1}.
\]
Using the additive notation
$K_{\cY}+D:=K_{\cY}\otimes\mathcal O_{\cY}(D)$, set
\begin{equation}\label{eq:L-log-f-D}
 L^{\log}_{f,D}:=f^*(K_{\cY}+D)\otimes K_{\cX}^{-1}.
\end{equation}
The superscript $\log$ refers to the logarithmic differential structure
determined by $D$. For a reduced divisor on a smooth curve, the sheaf
$\Omega^1_{\cY}(\log D)$ consists of meromorphic differentials regular away
from $D$ and having at most simple poles along $D$. Thus
\[
 \Omega^1_{\cY}(\log D)\simeq K_{\cY}(D).
\]
If $y\in\supp D$ and $w$ is a local parameter with $D=(w=0)$ locally,
this sheaf has local frame
\[
 \frac{dw}{w}=d\log w.
\]
Here $d\log w$ denotes $w^{-1}dw$, without requiring an analytic logarithm.
Hence $K_{\cY}+D$ is the logarithmic differential line bundle, and
$L^{\log}_{f,D}$ compares its pullback with $K_{\cX}$.
The superscript also distinguishes it from the first-main-theorem line
bundle $L_{f,D}=f^*\mathcal O_{\cY}(D)$.
Thus $\sigma_{f,D}$ is a meromorphic section of $L^{\log}_{f,D}$; it need
not be regular, since $J_f$ may vanish at ramification points.
\end{definition}

\begin{lemma}\label{lem:div-sigma}
With the preceding notation,
\[
        \operatorname{div}(\sigma_{f,D})=f^*D-R_f.
\]
\end{lemma}

\begin{proof}
Divisors of meromorphic sections are additive under tensor product and
subtractive under quotient.  Since
\[
        \operatorname{div}(f^*s_D)=f^*D,
        \qquad
        \operatorname{div}(J_f)=R_f,
\]
we obtain
\[
        \operatorname{div}(\sigma_{f,D})=f^*D-R_f.
\]
\end{proof}

\subsection{Skeletal identities for the second main theorem}

Equip $L^{\log}_{f,D}$ with a fixed tail-normalized admissible metric and set
\[
        u^{\log}_{f,D}:=(-\log\|\sigma_{f,D}\|)|_\Sigma.
\]
After subdivision, $u^{\log}_{f,D}$ is finite, continuous, and piecewise
affine. Define
\[
        m_f^{\log}(D,R)
        :=\int_{\partial\Sigma_R}u^{\log}_{f,D}\,d\nu_R
        =\sum_{h\in H}\lambda_h(R)u^{\log}_{f,D}(\xi_h(R)),
\]
where
\[
        \nu_R=\sum_{h\in H}\lambda_h(R)\delta_{\xi_h(R)}.
\]

Fix tail-normalized admissible metrics on $f^*(K_{\cY}+D)$,
$K_{\cX}$, $\mathcal O_{\cX}(H)$, and $K_{\cX}+H$.
By Lemmas~\ref{lem:metric-change} and~\ref{lem:T-additivity}, changing
these fixed metrics affects only the $O(1)$ terms.

\begin{theorem}\label{thm:smt-identity}\label{prop:skeleton-ramification}
Let $f:\cX\to\cY$ be a finite generically \`etale analytic morphism between
proper connected quasi-smooth strictly $K$-analytic Berkovich curves. 
Choose $R_0>0$ larger than every stabilization threshold determined by the
divisors, metrics, tail frames, and breakpoints appearing above.

\textup{(i)} 
For every $R\ge R_0$, put\ $\beta_h:=\ord_h(\sigma_{f,D})\qquad(h\in H).$
Then, on $\Sigma_R$,
\begin{equation}\label{eq:skeleton-ramification}
\begin{aligned}
        \Delta_{\Sigma_R}(u^{\log}_{f,D}|_{\Sigma_R})
        &=\omega_{L^{\log}_{f,D},\Sigma}|_{\Sigma_R}
        -\tau_{R*}\bigl(f^*D|_U\bigr)
        +\tau_{R*}\bigl(R_f|_U\bigr) -\sum_{h\in H}\beta_h\,\delta_{\xi_h(R)}.
\end{aligned}
\end{equation}

\medskip
\textup{(ii)} 
For every $R\ge R_0$,
\begin{equation}\label{eq:smt-identity-L}
        T_{L^{\log}_{f,D}}(R)+N_{R_f}(R)
        =N_f(D,R)+m_f^{\log}(D,R)-n_Hu^{\log}_{f,D}(b),
\end{equation}
and, by Lemma~\ref{lem:T-additivity},
\begin{equation}\label{eq:smt-identity-target-expanded}
        T_f(K_{\cY}+D,R)-T_{K_{\cX}}(R)+N_{R_f}(R)
        =N_f(D,R)+m_f^{\log}(D,R)-n_Hu^{\log}_{f,D}(b)+O(1).
\end{equation}
Here the $O(1)$ term may depend on the fixed metrics but is independent of $R$.
\end{theorem}

\begin{proof}(i).
By \eqref{eq:skeleton-PL-detailed} and Lemma~\ref{lem:div-sigma},
\[
\begin{aligned}
 \Delta_{\Sigma_R}(u^{\log}_{f,D}|_{\Sigma_R})
 &=\omega_{L^{\log}_{f,D},\Sigma}|_{\Sigma_R}
   -\tau_{R*}\bigl(\operatorname{div}(\sigma_{f,D})|_U\bigr)
   -\sum_{h\in H}\beta_h\delta_{\xi_h(R)}\\
 &=\omega_{L^{\log}_{f,D},\Sigma}|_{\Sigma_R}
   -\tau_{R*}(f^*D|_U)
   +\tau_{R*}(R_f|_U)
   -\sum_{h\in H}\beta_h\delta_{\xi_h(R)}.
\end{aligned}
\]

(ii).
Fix $R\ge R_0$. The graph $\Sigma_R$ is compact, and $G_R$ and
$u^{\log}_{f,D}|_{\Sigma_R}$ are continuous piecewise-affine functions
on a common finite subdivision. Lemma~\ref{lem:graph-ibp-full} applies to part (i).
Since $G_R(\xi_h(R))=0$, the endpoint atoms in part~\textup{(i)} contribute
zero. Using
\[
        \Delta_{\Sigma_R}G_R=\nu_R-n_H\delta_b
\]
and integration by parts gives
\[
\begin{aligned}
 &T_{L^{\log}_{f,D}}(R)-N_f(D,R)+N_{R_f}(R)\\
 &\quad=\int_{\Sigma_R}G_R\,d(\Delta_{\Sigma_R}u^{\log}_{f,D})\\
 &\quad=\int_{\Sigma_R}u^{\log}_{f,D}\,d(\Delta_{\Sigma_R}G_R)\\
 &\quad=\int_{\partial\Sigma_R}u^{\log}_{f,D}\,d\nu_R
        -n_Hu^{\log}_{f,D}(b)\\
 &\quad=m_f^{\log}(D,R)-n_Hu^{\log}_{f,D}(b).
\end{aligned}
\]
Lemma~\ref{lem:T-additivity}, applied to the fixed metrics, gives
\[
\begin{aligned}
 T_{L^{\log}_{f,D}}(R)
 &=T_{f^*(K_{\cY}+D)\otimes K_{\cX}^{-1}}(R)\\
  &=T_f(K_{\cY}+D,R)-T_{K_{\cX}}(R)+O(1).
\end{aligned}
\]
\end{proof}

\subsection{Asymptotic estimates at punctures}

For $h\in H$, let $e_h$ be the local degree of $f$ at $h$ and let
\[
        r_h:=\ord_h(R_f).
\]
Set
\[
        \epsilon_h(D)=
        \begin{cases}
        1,& f(h)\in\supp D,\\
        0,& f(h)\notin\supp D.
        \end{cases}
\]
We also write
\[
        H_D:=\{h\in H:\ f(h)\in\supp D\}.
\]

\begin{proposition}\label{prop:m-estimate}\label{lem:tail-slope}
With the preceding notation, fix a tail-normalized admissible metric on
$L^{\log}_{f,D}$ and set
\[
 u^{\log}_{f,D}:=(-\log\|\sigma_{f,D}\|)|_\Sigma.
\]
Choose $R_0$ larger than every stabilization threshold determined by the fixed
divisors, metrics, and terminal slopes.

\textup{(i)} For every $h\in H$ and every $R\ge R_0$,
\begin{equation}\label{eq:tail-slope}
        u^{\log}_{f,D}(\xi_h(R))=\alpha_h R+O(1),
        \qquad
        \alpha_h:=\epsilon_h(D)e_h-r_h,
\end{equation}
and
\begin{equation}\label{eq:alpha-bound}
        \alpha_h\leq \epsilon_h(D)\leq 1.
\end{equation}

\medskip
\textup{(ii)} For every $R\ge R_0$,
\begin{equation}\label{eq:m-linear}
        m_f^{\log}(D,R)
        =\left(\sum_{h\in H}\alpha_h\right)R+O(1).
\end{equation}
In particular,
\begin{equation}\label{eq:m-bound-HD}
        m_f^{\log}(D,R)\leq |H_D|R+O(1).
\end{equation}
Moreover, if $T_H(R)$ is the characteristic of the reduced puncture divisor
$H=\sum_{h\in H}[h]$, then there is a constant $C_{f,D}$, independent of
$R$, such that
\begin{equation}\label{eq:m-minus-H}
        m_f^{\log}(D,R)-T_H(R)\leq C_{f,D}
        \qquad(R\ge R_0).
\end{equation}
\end{proposition}

\begin{proof}(i).
Apply the same completed-local-ring description used in
Lemma~\ref{lem:local-ramification} at the puncture $h$.  Thus, for local
parameters $t$ at $h$ and $y$ at $f(h)$,
\[
        f^*y=u(t)t^{e_h},
        \qquad u(0)\ne0.
\]
Since \(D\) is reduced,
\[
        \ord_h(f^*D)=\epsilon_h(D)e_h.
\]
By definition,
\[
        \ord_h(J_f)=r_h.
\]
Thus
\[
        \ord_h(\sigma_{f,D})
        =\ord_h(f^*s_D)-\ord_h(J_f)
        =\epsilon_h(D)e_h-r_h=\alpha_h.
\]
In a tail-normalized local frame \(e\) of \(L^{\log}_{f,D}\), we may write
\[
        \sigma_{f,D}=t^{\alpha_h}u_1(t)e,
        \qquad u_1(0)\ne0.
\]
Hence, on the corresponding ray,
\[
\begin{aligned}
        u^{\log}_{f,D}
        &=-\log\|\sigma_{f,D}\|  \\
        &=\alpha_h(-\log|t|)-\log|u_1(t)|-\log\|e\|  \\
        &=\alpha_h x_h+c'_h,
\end{aligned}
\]
on a sufficiently far part of the tail.  Here \(c'_h\in\mathbb R\) is
constant: the chosen tail coordinate differs from \(-\log|t|\) by a
constant, the analytic unit has constant absolute value on a sufficiently
small punctured disk, and the tail-normalized frame has norm one.  Evaluation
at \(x_h(\xi_h(R))=R\) proves \eqref{eq:tail-slope}.  The local ramification
estimate gives \(r_h\ge e_h-1\).  If \(\epsilon_h(D)=0\), then
\(\alpha_h=-r_h\le0\).  If \(\epsilon_h(D)=1\), then
\[
        \alpha_h=e_h-r_h\le e_h-(e_h-1)=1.
\]
Thus \(\alpha_h\le\epsilon_h(D)\le1\).

(ii).
Write
\[
 \lambda_h(R)=1+\varepsilon_h(R),
 \qquad
 u^{\log}_{f,D}(\xi_h(R))=\alpha_hR+\eta_h(R),
\]
where, by Lemma~\ref{lem:green-asymptotics} and part~\textup{(i)},
\[
 \varepsilon_h(R)=O(R^{-1}),
 \qquad
 \eta_h(R)=O(1).
\]
Here
\[
 \alpha_h=\epsilon_h(D)e_h-r_h
\]
is independent of \(R\).  Substituting these expressions into the definition of the boundary term gives
\[
\begin{aligned}
 m_f^{\log}(D,R)
 &=\sum_{h\in H}\lambda_h(R)u^{\log}_{f,D}(\xi_h(R))\\
 &=\sum_{h\in H}(1+\varepsilon_h(R))(\alpha_hR+\eta_h(R))\\
 &=R\sum_{h\in H}\alpha_h
   +\sum_{h\in H}
   \left(
      \eta_h(R)
      +\varepsilon_h(R)\alpha_hR
      +\varepsilon_h(R)\eta_h(R)
   \right).
\end{aligned}\]
For each fixed \(h\in H\),
\[
 \eta_h(R)=O(1),
 \qquad
 \varepsilon_h(R)\alpha_hR=O(1),
 \qquad
 \varepsilon_h(R)\eta_h(R)=O(R^{-1}),
\]
since \(\alpha_h\) is independent of \(R\).  As \(H\) is finite, the sum of these terms is again \(O(1)\).  Hence
\[
 m_f^{\log}(D,R)=\left(\sum_{h\in H}\alpha_h\right)R+O(1),
\]
which proves \eqref{eq:m-linear}.  Since
$\alpha_h\leq\epsilon_h(D)$,
\[
        \sum_{h\in H}\alpha_h\leq\sum_{h\in H}\epsilon_h(D)=|H_D|,
\]
Lemma~\ref{lem:T-degree} gives $T_H(R)=n_HR+O(1)$.  Consequently,
\[
\begin{aligned}
 m_f^{\log}(D,R)-T_H(R)
 &=\left(\sum_{h\in H}\alpha_h-n_H\right)R+E(R),
\end{aligned}
\]
where \(|E(R)|\le C_0\) for a constant \(C_0\) independent of \(R\).
Because \(|H_D|\le n_H\),
\[
 \sum_{h\in H}\alpha_h-n_H
 \le |H_D|-n_H\le0.
\]
It follows that \(m_f^{\log}(D,R)-T_H(R)\le C_0\) for \(R\ge R_0\), so one may
take \(C_{f,D}=C_0\).
\end{proof}
\subsection{The second main theorem for morphisms}\label{sec:truncated-form}\label{sec:truncated-smt}

We use the local ramification inequality to convert the ramified identity into level-one counting as follows.

\begin{lemma}\label{lem:truncation}
Let $f:\cX\to\cY$ be a finite generically \`etale analytic morphism. Denote
\[
 N_{R_f}(R)=\int_{\Sigma_R}G_R\,d\tau_{R*}(R_f|_U).
\]
Then, for every $R>0$,
\begin{equation}\label{eq:truncation-measure}
        N_f(D,R)-N_{R_f}(R)
        \leq N_f^{[1]}(D,R).
\end{equation}
\end{lemma}

\begin{proof}
Fix a type-I point $x\in \cX\setminus H$, and write
$r_x:=\ord_x(R_f).$
Because $D$ is reduced, its coefficient at every point of $\supp D$ is
one.  Hence, if $f(x)\in\supp D$, the pullback divisor has coefficient
\[
        \ord_x(f^*D)=e_x,
\]
where $e_x$ is the local degree of $f$ at $x$.  Lemma~\ref{lem:local-ramification}
gives $r_x\geq e_x-1$, and therefore
\[
        \ord_x(f^*D-R_f)=e_x-r_x\leq1.
\]
On the other hand, if $f(x)\notin\supp D$, then
$\ord_x(f^*D)=0$, so
\[
        \ord_x(f^*D-R_f)=-r_x\leq0,
\]
since the ramification divisor $R_f=\operatorname{div}(J_f)$ is effective.  Thus, in
both cases,
\[
 \ord_x(f^*D-R_f)
 \leq
 \ord_x\bigl((f^*D)_{\red}\bigr).
\]
Restricting to $U$ yields the coefficientwise divisor inequality
\[
        f^*D|_U-R_f|_U\leq (f^*D|_U)_{\red}.
\]
Equivalently, the difference
\[
        (f^*D|_U)_{\red}-\bigl(f^*D|_U-R_f|_U\bigr)
\]
is an effective divisor on $U$.  Pushforward by $\tau_R$ preserves
positivity of atomic measures, so
\[
 \tau_{R*}\bigl((f^*D-R_f)|_U\bigr)
 \leq
 \tau_{R*}\bigl((f^*D|_U)_{\red}\bigr).
\]
Finally, Lemma~\ref{lem:green-asymptotics} gives $G_R\geq0$.  Integrating
this measure inequality against the nonnegative function $G_R$ preserves
the order, and hence
\[
\begin{aligned}
 N_f(D,R)-N_{R_f}(R)
 &=\int_{\Sigma_R}G_R\,
   d\tau_{R*}\bigl((f^*D-R_f)|_U\bigr)\\
 &\leq
 \int_{\Sigma_R}G_R\,
   d\tau_{R*}\bigl((f^*D|_U)_{\red}\bigr)\\
 &=N_f^{[1]}(D,R).
\end{aligned}
\]
\end{proof}

Below we introduce our second main theorem.

\begin{theorem}[Second main theorem for morphisms of Berkovich curves]\label{thm:target-curve-SMT}\label{cor:target-truncated-SMT}
Let $f:\cX\to\cY$ be a finite generically \`etale analytic morphism between
proper connected quasi-smooth strictly $K$-analytic curves. Let
$H\subset\cX(K)$ be a nonempty finite set, put $U:=\cX\setminus H$, and
let $R_f:=\operatorname{div}(J_f)$. Let $D$ be a reduced effective analytic
Cartier divisor on $\cY$, and fix tail-normalized admissible metrics on
$f^*(K_{\cY}+D)$ and $K_{\cX}+H$. Use the fixed skeleton and Green kernels
of Theorem~\ref{thm:smt-identity}, and the counting functions of
Lemma~\ref{lem:truncation}. Choose a common $R_0>0$ beyond all the
stabilization thresholds in that theorem.

\textup{(i)} For every $R\ge R_0$,
\begin{equation}\label{eq:target-curve-SMT}
        T_f(K_{\cY}+D,R)-T_{K_{\cX}+H}(R)+N_{R_f}(R)
        \leq N_f(D,R)+O(1).
\end{equation}
The constant in $O(1)$ depends on the chosen skeleton, tail coordinates, base point, metrics, $f$, and $D$, but not on $R$.

\medskip
\textup{(ii)} For every $R\ge R_0$,
\begin{equation}\label{eq:target-truncated-SMT}
        T_f(K_{\cY}+D,R)-T_{K_{\cX}+H}(R)
        \leq N_f^{[1]}(D,R)+O(1)
\end{equation}
\end{theorem}

\begin{proof}(i). For arbitrary fixed tail-normalized admissible
metrics, it follows from the second main identity
\eqref{eq:smt-identity-target-expanded} in Theorem \ref{prop:skeleton-ramification} that 
\[
\begin{aligned}
 T_f(K_{\cY}+D,R)-T_{K_{\cX}}(R)+N_{R_f}(R)
 &=N_f(D,R)+m_f^{\log}(D,R)\\
 &\quad-n_Hu^{\log}_{f,D}(b)+O(1),
\end{aligned}
\]
where the $O(1)$ term is independent of $R$.  By Lemma~\ref{lem:T-additivity},
applied to
\[
        K_{\cX}+H=K_{\cX}\otimes\mathcal O_{\cX}(H),
\]
one has
\[
        T_{K_{\cX}+H}(R)=T_{K_{\cX}}(R)+T_H(R)+O(1).
\]
Substituting this into the preceding identity and collecting the bounded
metric-comparison terms gives
\[
\begin{aligned}
 &T_f(K_{\cY}+D,R)-T_{K_{\cX}+H}(R)+N_{R_f}(R)\\
 &\qquad=N_f(D,R)
   +\bigl(m_f^{\log}(D,R)-T_H(R)\bigr)
   -n_Hu^{\log}_{f,D}(b)+O(1).
\end{aligned}
\]
By Proposition~\ref{prop:m-estimate}, there is a constant $C_{f,D}$,
independent of $R$, such that
\[
        m_f^{\log}(D,R)-T_H(R)\leq C_{f,D}.
\]
Moreover, the base point $b$, the section $\sigma_{f,D}$, and the chosen
metric are fixed, so $u^{\log}_{f,D}(b)$ is independent of $R$.  Thus both
$C_{f,D}$ and $-n_Hu^{\log}_{f,D}(b)$ may be absorbed into the bounded error term.
Consequently,
\[
        T_f(K_{\cY}+D,R)-T_{K_{\cX}+H}(R)+N_{R_f}(R)
        \leq N_f(D,R)+O(1),
\]
which proves \eqref{eq:target-curve-SMT}.

(ii).
Lemma~\ref{lem:truncation} gives the pointwise-truncation estimate
\[
        N_f(D,R)-N_{R_f}(R)\le N_f^{[1]}(D,R).
\]
Substituting this into part~\textup{(i)} therefore gives
\[
        T_f(K_{\cY}+D,R)-T_{K_{\cX}+H}(R)
        \le N_f^{[1]}(D,R)+O(1),
\]
which is exactly \eqref{eq:target-truncated-SMT}.
\end{proof}

\begin{definition}[Normalized characteristic function]\label{def:normalized-characteristic}
Let
\[
        f:\cX\longrightarrow\cY
\]
be a finite generically \'etale analytic morphism.
Fix a point $a_0\in\cY $, and let $[a_0]$ be the associated reduced
effective Cartier divisor of degree one. Equip the pullback line bundle
\[
        L_{f,[a_0]}:=f^*\cO_{\cY}([a_0])
\]
with a fixed tail-normalized admissible metric relative to $(\Sigma,H)$,
in the sense of Convention~\ref{conv:tail-normalization}. The \emph{normalized characteristic
function} of $f$ is
\begin{equation}\label{eq:normalized-characteristic}
        T_f(R):=T_f([a_0],R).
\end{equation} In particular, one can take $[a_0]=[1]$ in $\mathbf P_K^{1,\an}.$
\end{definition}

\begin{remark}
The definition depends on the auxiliary choices of $a_0$ and of the
tail-normalized admissible metric on $L_{f,[a_0]}$. However, for any two
such choices, the resulting functions differ by $O(1)$ as $R\to\infty$.
Indeed, by Lemma~\ref{lem:metric-change}, changing the metric on the fixed
line bundle $L_{f,[a_0]}$ changes $T_f(R)$ by $O(1)$. If $a_0'\in\cY(K)$
is another point, then $\cO_{\cY}([a_0'])\simeq\cO_{\cY}([a_0])$ because
both are degree-one line bundles on the curve $\cY$, and any isomorphism
between them, together with a fixed metric on the source, induces a
tail-normalized admissible metric on $L_{f,[a_0']}$ whose characteristic
differs from $T_f(R)$ by $O(1)$ by the same lemma. Consequently, the
linear growth
\[
        T_f(R)=\deg(f)R+O(1)
        \qquad(R\to\infty)
\]
is independent of these choices; it is the content of
Corollary~\ref{cor:morphism-characteristic-growth}. In particular,
$T_f(R)>0$ for all sufficiently large $R$, and $T_f(R)\to+\infty$, so
$T_f(R)$ may be used as the normalizing denominator in the defect definitions in the next subsection.
\end{remark}

\begin{corollary}[Second main theorem for Projective line]\label{cor:log-canonical-SMT}\label{cor:truncated-SMT}
Let
$f:\cX\to\mathbf P_K^{1,\an}$ be a finite generically \`etale analytic
morphism, let $R_f:=\operatorname{div}(J_f)$, and let
$D:=a_1+\cdots+a_q$ be a reduced divisor supported on distinct points
$a_j\in\mathbf P^1(K)$.  
Retain the preceding metric conventions and choose a common stabilization
threshold $R_0$ for these fixed choices.

\textup{(i)} For every $R\ge R_0$,
\begin{equation}\label{eq:log-canonical-SMT}
        (q-2)T_f(R)-T_{K_{\cX}+H}(R)+N_{R_f}(R)
        \leq \sum_{j=1}^q N_f(a_j,R)+O(1).
\end{equation}

\medskip
\textup{(ii)} For every $R\ge R_0$,
\begin{equation}\label{eq:truncated-SMT}
        (q-2)T_f(R)-T_{K_{\cX}+H}(R)
        \leq \sum_{j=1}^q N_f^{[1]}(a_j,R)+O(1).
\end{equation}
\end{corollary}
\begin{proof} 
Since $K_{\mathbf P_K^{1,\an}}\simeq\mathcal O_{\mathbf P_K^{1,\an}}(-2)$
and $\mathcal O_{\mathbf P_K^{1,\an}}(D)\simeq\mathcal O_{\mathbf P_K^{1,\an}}(q)$,
we have
\[
        K_{\mathbf P_K^{1,\an}}+D
        \simeq\mathcal O_{\mathbf P_K^{1,\an}}(q-2).
\]
After pullback,
\[
        f^*(K_{\mathbf P_K^{1,\an}}+D)
        \simeq\bigl(f^*\mathcal O_{\mathbf P_K^{1,\an}}(1)\bigr)^{\otimes(q-2)}.
\]
 By Lemma~\ref{lem:T-additivity}, together with the
bounded metric-comparison of Lemma~\ref{lem:metric-change},
\[
        T_f(K_{\mathbf P_K^{1,\an}}+D,R)
        =(q-2)\,T_f(R)+O(1),
\]
where $T_f(R)=T_f(\mathcal O_{\mathbf P_K^{1,\an}}(1),R)$ is the
normalized characteristic of Definition~\ref{def:normalized-characteristic}.

Because $a_1,\dots,a_q$ are distinct, the fibres
$f^{-1}(a_1),\dots,f^{-1}(a_q)$ are pairwise disjoint. Hence the pullback
divisors $f^*[a_1],\dots,f^*[a_q]$ have pairwise disjoint supports, and so
do their reductions modulo multiplicity. Therefore
\[
        f^*D=\sum_{j=1}^q f^*[a_j],
        \qquad
        (f^*D|_U)_{\mathrm{red}}
        =\sum_{j=1}^q\bigl(f^*[a_j]|_U\bigr)_{\mathrm{red}}.
\]
By linearity of the pushforward $\tau_{R*}$ and of the Green pairing
against $G_R$, we obtain
\[
        N_f(D,R)=\sum_{j=1}^q N_f(a_j,R),
        \qquad
        N_f^{[1]}(D,R)=\sum_{j=1}^q N_f^{[1]}(a_j,R).
\]
Then substituting the above identities  into
Theorem~\ref{thm:target-curve-SMT} gives \eqref{eq:log-canonical-SMT} and \eqref{eq:truncated-SMT}. 
\end{proof}

\subsection{Defect relations}

\begin{definition}[Defect]\label{def:defect} Let $f:\cX\to\cY$ be a finite generically \'etale analytic morphism 
between proper connected quasi-smooth strictly $K$-analytic curves.  The defect and reduced-defect of a nonzero effective analytic Cartier divisor $D$ on $\cY$,  with respect to $f,$ are defined by
\begin{equation}\label{eq:defect-definition}
 \delta_f(D):=1-\limsup_{R\to\infty}
 \frac{N_f(D,R)}{\deg(D)T_f(R)}, \quad \delta^{[1]}_f(D):=1-\limsup_{R\to\infty}
 \frac{N^{[1]}_f(D,R)}{\deg(D)T_f(R)},
\end{equation}   respectively.
\end{definition}

\begin{theorem}[Defect relation]\label{thm:defect-relation}Under the assumption of Theorem \ref{cor:target-truncated-SMT}, we have the following inequalities:
\begin{equation}\label{eq:defect-relation}
 \delta_f(D)\le \frac{1}{\deg(D)}\left(
 2-2g(\cY)+\frac{2g(\cX)-2+n_H}{\deg(f)}
 -\frac{\deg(R_f|_U)}{\deg(f)}\right)
\end{equation}
and \begin{equation}\label{eq:reduced-defect-relation}
 \delta^{[1]}_f(D)\le \frac{1}{\deg(D)}\left(
 2-2g(\cY)+\frac{2g(\cX)-2+n_H}{\deg(f)}\right).
\end{equation}
\end{theorem}

\begin{proof}
By Theorem~\ref{thm:target-curve-SMT}, for every sufficiently large $R$,
\[
 T_f(K_{\cY}+D,R)-T_{K_{\cX}+H}(R)+N_{R_f}(R)
 \le N_f(D,R)+O(1).
\]
Divide both sides by $\deg(D)\,T_f(R)$ and take the limit superior as
$R\to\infty$. By Lemma~\ref{lem:T-degree},
Corollary~\ref{cor:morphism-characteristic-growth}, and the definition of
the normalized characteristic $T_f(R)=T_f([a_0],R)$ in
Definition~\ref{def:normalized-characteristic}, we have
\[
 \frac{T_f(K_{\cY}+D,R)}{\deg(D)\,T_f(R)}
 =\frac{\deg(f)(2g(\cY)-2+\deg D)R+O(1)}{\deg(D)\cdot \deg(f)R+O(1)}
 \longrightarrow \frac{2g(\cY)-2+\deg D}{\deg D},
\]
\[
 \frac{T_{K_{\cX}+H}(R)}{\deg(D)\,T_f(R)}
 \longrightarrow \frac{2g(\cX)-2+n_H}{\deg(f)\deg D},
\]
and
\[
 \frac{N_{R_f}(R)}{\deg(D)\,T_f(R)}
 \longrightarrow \frac{\deg(R_f|_U)}{\deg(f)\deg D}.
\]
Therefore
\[
 \frac{2g(\cY)-2+\deg D}{\deg D}
 -\frac{2g(\cX)-2+n_H}{\deg(f)\deg D}
 +\frac{\deg(R_f|_U)}{\deg(f)\deg D}
 \le 1-\delta_f(D).
\]
Since
\[
 1-\frac{2g(\cY)-2+\deg D}{\deg D}
 =\frac{2-2g(\cY)}{\deg D},
\]
rearranging yields \eqref{eq:defect-relation}. Using the level-one truncated counting function and the conclusion (ii) in  Theorem~\ref{thm:target-curve-SMT} we also obtain \eqref{eq:reduced-defect-relation}.
\end{proof}

\end{document}